\documentclass[final]{siamltex}
\usepackage{amssymb}
\usepackage{graphicx}
\usepackage[numbers,sort&compress]{natbib}

\title{A study on perturbation analysis of spectral preconditioners}

\author{Tao Zhao\thanks{Laboratoire J.L. Lions, University of Paris VI({\tt zhaot.cn@gmail.com}).}}

\begin{document}
\maketitle

\begin{abstract}
It is well-known that the convergence of Krylov subspace methods to solve linear system depends on the spectrum of the coefficient matrix, moreover, it is widely accepted that for both symmetric and unsymmetric systems Krylov subspace methods will converge fast if the spectrum of the coefficient matrix is clustered. In this paper we investigate the spectrum of the system preconditioned by the deflation, coarse grid correction and adapted deflation preconditioners. Our analysis shows that the spectrum of the preconditioned system is highly impacted by the angle between the coarse grid space for the construction of the three preconditioners and the subspace spanned by the eigenvectors associated with the small eigenvalues of the coefficient matrix. Furthermore, we prove that with a certain restriction the accuracy of the inverse of projection matrix also impacts the spectrum of the preconditioned system. Numerical experiments emphasized the theoretical analysis.
\end{abstract}



\pagestyle{myheadings}
\thispagestyle{plain}

\section{Introduction}
We consider the iterative solution of the linear system
$$Ax=b,$$
where $A\in\mathbb{R}^{n\times n}$ is symmetric positive definite(SPD). It is well known that the convergence of Krylov subspace method for solving linear system highly depends on the eigenvalue distribution. Recently many studies\cite{NatafHuaVic, NatafMikolajZhao, TangNVE, deSturler, MorganDeflation, Padiy, ErlanggaSisc, ErlanggaSimax, NabbenVuik, Giraud, Gosselet} have shown that removing the very small eigenvalues makes the spectrum more clustered and consequently the convergence is improved. Preconditioners based on deflation technique and coarse grid correction are two main types of the spectral preconditioners that deflate or shift small eigenvalues.

To describe the preconditioners discussed in the following sections, we define
\begin{eqnarray}\label{de:projection}
E=Z^TAZ,~~Z\in\mathbb{R}^{n\times r}.
\end{eqnarray}
The matrix of the form (\ref{de:projection}) is referred to as projection matrix. The subspace spanned by the columns of $Z$ is referred to as coarse grid space in domain decomposition method or deflation space in deflation method. In ideal case, the coarse grid space or deflation space should contain the vectors corresponding to the lower part of the spectrum that are responsible for the stagnation in convergence. Section 2 will show that the preconditioned systems have good properties when the preconditioners are constructed with the ideal subspace. For simplicity, we refer to the subspace spanned by the columns of $Z$ as coarse grid space throughout the paper. In particular, we refer to the coarse grid space spanned by the eigenvectors associated with the small eigenvalues of $A$ as the exact coarse grid space.

In \cite{ErlanggaSimax,NabbenVuik}, the deflation preconditioner is defined by
\begin{eqnarray}\label{de:P_D}
P_D=I-AZE^{-1}Z^{T}.
\end{eqnarray}
Obviously $P_DA$ is singular since $P_D$ is singular. Fortunately, Krylov subspace methods can still solve singular linear system as long as it is consistent; furthermore, zero eigenvalues do not impact the convergence since the corresponding eigenvectors never enter the iteration. To measure the convergence rate of Krylov subspace methods for symmetric positive semidefinite system, the effective condition number\cite{NabbenVuik}, $\kappa_{eff}$, is defined by the ratio of the maximal to the minimal nonzero eigenvalues.

In contrast to the deflation method, the preconditioners based on coarse grid correction shift the small eigenvalues to the large ones. In \cite{TangNVE}, a well-known coarse grid correction preconditioner in domain decomposition method is defined by
\begin{eqnarray}\label{de:P_C}
P_C=I+ZE^{-1}Z^{T}.
\end{eqnarray}
The abstract additive coarse-grid correction is $M^{-1}+ZE^{-1}Z^T$, where $M$ is the sum of the local solves in each subdomain. Another well-known preconditioner in domain decomposition method is the abstract balancing preconditioner\cite{ErlanggaSimax}
\begin{eqnarray}\label{de:P_BNN}
P_{BNN}=(I-ZE^{-1}Z^{T}A)M^{-1}(I-AZE^{-1}Z^{T})+ZE^{-1}Z^{T}.
\end{eqnarray}
The adapted deflation preconditioner\cite{TangNVE} is defined by
\begin{eqnarray*}
P_{ADEF1}=M^{-1}P_D+ZE^{-1}Z^{T}.
\end{eqnarray*}
It has been shown in \cite{TangNVE} that $P_{ADEF1}$ is less expensive than $P_{BNN}$ but as robust as $P_{BNN}$. Moreover, $P_{BNN}A$ and $P_{ADEF1}A$ have identical spectrum. In both $P_{BNN}$ and $P_{ADEF1}$, the first item is responsible for the fine grid and the second item term for the coarse grid. So they are called two-level preconditioners. In this paper, we confine our analysis on one-level methods. Thus we let $M=I$ and define $P_A$ by
\begin{eqnarray}\label{de:P_A-DEF1}
P_A=I-AZE^{-1}Z^T+ZE^{-1}Z^{T}.
\end{eqnarray}

Let $X$ be a basis of coarse grid space. Obviously the following equation
$$X(X^TAX)^{-1}X^T=Z(Z^TAZ)^{-1}Z^T$$
holds in exact arithmetic since there exists a nonsingular matrix $C\in\mathbb{R}^{r\times r}$ such that $X=ZC$. That means $P_D$, $P_C$ and $P_A$ are determined uniquely by coarse grid space. Thus we can choose an appropriate basis to form $Z$ and then construct preconditioners. In section 3, we will show this fact is helpful for the analysis.

The paper is organized as follows. In section 2, we briefly review the spectral properties of the preconditioned system when all preconditioners are constructed with the exact coarse grid space and inverse of projection matrix. Section 3 presents the perturbation analysis on the spectrum of the preconditioned system when the three preconditioners are constructed with the coarse grid space spanned by the approximate eigenvectors. Section 4 anlyses the spectrum of the preconditioned system when the inverse of the projection matrix are solved inexactly. Numerical results are presented in Section 5.

\section{Coarse grid subspace spanned by the exact eigenvectors}\label{se:spectrum}
Let $(\lambda_i, v_i)$ be an eigenpair of $A$ and $v_i$ be normalized. $(v_i,\cdots,v_n)$ is orthogonal since $A$ is SPD. The spectral decomposition of $A$ can be written as
\begin{eqnarray}
A=(V,V_\perp)\left(\begin{array}{cc}\Lambda & 0\\ 0 & \Lambda_\perp\end{array}\right)\left(\begin{array}{c}V^T\\V^T_\perp\end{array}\right),
\end{eqnarray}
where $\Lambda=diag\{\lambda_1,\cdots,\lambda_r\}$, $\Lambda_\perp=diag\{\lambda_{r+1},\cdots,\lambda_n\}$, $V=(v_1,\cdots,v_r)$, and $V_\perp=(v_{r+1},\cdots,v_n)$. In this section, we assume that the eigenvalues of $\Lambda$ are the small ones to be removed and that the three preconditioners are constructed with $Z=V$ and the exact inverse of projection matrix.

\begin{theorem}\label{th:tP_D}
Let
\begin{eqnarray}\label{de:tP_D}
\tilde P_D=I-AV\tilde E^{-1}V^T,~~~\tilde E=V^TAV.
\end{eqnarray}
Then $\tilde P_DA$ is symmetric and has the following spectral decomposition
\begin{eqnarray*}
\tilde P_DA=(V,V_\perp)\left(\begin{array}{cc} 0 & 0\\ 0 & \Lambda_\perp\end{array}\right)\left(\begin{array}{c}V^T\\V^T_\perp\end{array}\right).\label{eq:tP_DA}
\end{eqnarray*}
In addition, $\tilde P_DA=A\tilde P_D$.
\end{theorem}
\begin{proof}
Obviously, $\tilde P_DA$ is symmetric and $\tilde P_DAV=0$. From the definition of $\tilde E$, we have
$$\tilde E=V^TAV=V^TV\Lambda=\Lambda.$$
Then we obtain
\begin{eqnarray*}
\tilde P_DAV_\perp &=& AV_\perp-AV\Lambda^{-1}V^TAV_\perp\\
                   &=& V_\perp\Lambda_\perp-VV^TV_\perp\Lambda_\perp\\
                   &=& V_\perp\Lambda_\perp.
\end{eqnarray*}
Thus the first result holds. Similarly, we have $A\tilde P_DV=AV-AV\Lambda\Lambda^{-1}V^TV=0$ and
$A\tilde P_DV_\perp=AV_\perp+AV\Lambda^{-1}V^TV_\perp=V_\perp\Lambda_\perp$.
Thus $\tilde P_DA=A\tilde P_D$.
\end{proof}
\begin{theorem}\label{th:tP_C}
Let
\begin{eqnarray}\label{de:tP_C}
\tilde P_C=I+V\tilde E^{-1}V^T,~~~\tilde E=V^TAV.
\end{eqnarray}
Then $\tilde P_CA$ is symmetric and has the following spectral decomposition
\begin{eqnarray*}
\tilde P_CA &=& (V,V_\perp)\left(\begin{array}{cc} I+\Lambda & 0\\ 0 & \Lambda_\perp\end{array}\right)\left(\begin{array}{c}V^T\\V^T_\perp\end{array}\right),\label{eq:tP_CA}
\end{eqnarray*}
In addition, $\tilde P_CA=A\tilde P_C$.
\end{theorem}
\begin{proof}
The proof is similar to that of Theorem \ref{th:tP_D}.
\end{proof}
\begin{theorem}\label{th:tP_ADEF1}
Let
\begin{eqnarray}\label{de:tP_ADEF1}
\tilde P_A=I-AV\tilde E^{-1}V^T+V\tilde E^{-1}V^T,~~~\tilde E=V^TAV.
\end{eqnarray}
Then $\tilde P_AA$ is symmetric and has the following spectral decomposition
\begin{eqnarray*}
\tilde P_AA &=& (V,V_\perp)\left(\begin{array}{cc} I & 0\\ 0 & \Lambda_\perp\end{array}\right)\left(\begin{array}{c}V^T\\V^T_\perp\end{array}\right),\label{eq:tP_A-DEF1A}
\end{eqnarray*}
In addition, $\tilde P_AA=A\tilde P_A$.
\end{theorem}
\begin{proof}
The proof is similar to that of Theorem \ref{th:tP_D}.
\end{proof}

Theorems \ref{th:tP_D}-\ref{th:tP_ADEF1} show that the small eigenvalues are removed no matter whether the preconditioners are applied on the left or right side of $A$, moreover, the rest eigenvalues and all eigenvectors are not changed. In addition, it follows from these theorems that $\kappa_{eff}(\tilde P_DA)\le\kappa(\tilde P_CA)$ and $\kappa_{eff}(\tilde P_DA)\le\kappa(\tilde P_AA)$, where $\kappa$ is the spectral condition number\cite{Saad}.

\section{Coarse grid subspace spanned by the approximate eigenvectors}
In practice for a large system, it is rather expensive to obtain the exact coarse grid space spanned by eigenvectors. An alternative to make the coarse grid space have rich information on eigenvectors is to build a large subspace\cite{NatafHuaVic,NatafMikolajZhao,Padiy,ErlanggaSisc}. However, this way will lead to a large projection matrix whose inverse is also expensive to solve exactly. So it is natural to construct preconditioners by replacing the exact coarse grid space by an approximate one and computing $E^{-1}$ approximately. We hope that the system preconditioned by the resulting preconditioners has the similar eigenvalue distribution as described in the previous section so that the convergence can be improved. That motivates us to analyse how the perturbation in the coarse grid space impacts the spectrum of the preconditioned system.
\if
and $E^{-1}$ impacts the spectrum of the preconditioned system. Inexact computation means that the coarse grid space is spanned by the approximate eigenvectors and that $E^{-1}$ is not solved accurately.
\fi

Assume $Z$ is column orthogonal in the following discussion. Let $\mathcal{Z}$ denote the subspace spanned by the columns of $Z$. Let $\mathcal{Z}^\perp$ be the orthogonal complement of $\mathcal{Z}$ and $Z_\perp$ be an orthogonal basis of $\mathcal{Z}^\perp$. Likewise, let $\mathcal{V}$ be the subspace spanned by the columns of $V$, $\mathcal{V}^\perp$ be the orthogonal complement of $\mathcal{V}$ and $V_\perp$ be an orthogonal basis of $\mathcal{V}^\perp$.

Let $\sigma$ denote the singular value of a matrix. Let $dist(\mathcal{Z},\mathcal{V})$ denote the distance between subspaces $\mathcal{Z}$ and $\mathcal{V}$. It has been proven in \cite[Theorem 2.6.1]{Gene} that $dist(\mathcal{Z},\mathcal{V})=\sigma_{max}(Z^TV_\perp)=\sigma_{max}(V^TZ_\perp)$. Note that $0\le dist(\mathcal{Z},\mathcal{V})\le1$ since $\mathcal{Z}$ and $\mathcal{V}$ have the same dimension. We define the acute angle between subspaces $\mathcal{Z}$ and $\mathcal{V}$ as $\theta=\arcsin dist(\mathcal{Z}, \mathcal{V})$. 
\begin{lemma}\label{le:angle}
Let $\theta$ be the acute angle between subspaces $\mathcal{Z}$ and $\mathcal{V}$ that have the same dimension. Let $Z$ and $V$ be the orthogonal bases of $\mathcal{Z}$ and $\mathcal{V}$ respectively. Then 
\begin{eqnarray*}
\sin\theta &=& \sigma_{max}(Z^TV_\perp)=\sigma_{max}(V^TZ_\perp),\\
\cos\theta &=& \sigma_{min}(Z^TV)=\sigma_{min}(Z^T_\perp V_\perp).
\end{eqnarray*}
\end{lemma}\begin{proof}
The proof of the first equation can be found in \cite[Theorem 2.6.1]{Gene}. 
Since $(V,V_\perp)$ is orthogonal, it follows from $\|(V,V_\perp)^TZx\|_2=1$ for all unit 2-norm $x\in\mathbb{R}^r$ that $\|V^TZx\|^2_2+\|V^T_\perp Zx\|^2_2=1.$ Thus
\begin{eqnarray*}
\sigma_{min}(V^TZ)^2 &=& \min_{\|x\|_2=1}\|V^TZx\|^2_2=1-\max_{\|x\|_2=1}\|V^T_\perp Zx\|^2_2\\
                     &=& 1-\sigma_{max}(V^T_\perp Z)^2=\cos^2\theta.
\end{eqnarray*}
Similarly, since $(Z,Z_\perp)$ is orthogonal, it follows from $\|(Z,Z_\perp)^TV_\perp x\|_2=1$ for all unit 2-norm $x\in\mathbb{R}^{(n-r)}$ that $\|Z^TV_\perp x\|^2_2+\|Z^T_\perp V_\perp x\|^2_2=1.$ Thus
\begin{eqnarray*}
\sigma_{min}(Z^T_\perp V_\perp)^2 &=& \min_{\|x\|_2=1}\|Z^T_\perp V_\perp x\|^2_2=1-\max_{\|x\|_2=1}\|Z^TV_\perp x\|^2_2\\
                                &=& 1-\sigma_{max}(Z^TV_\perp)^2=\cos^2\theta.
\end{eqnarray*}
Thus, $\cos\theta=\sigma_{min}(Z^TV)=\sigma_{min}(Z^T_\perp V_\perp)$.
\end{proof}

Let $P_D$ be defined by (\ref{de:P_D}). $P_DA$ is an symmetric matrix since $A$ is SPD. Hence, Courant-Fischer Minimax theorem\cite{Gene} can be applied to estimate the eigenvalues of $P_DA$, which is described in Theorem \ref{th:CF-mm}.
\begin{theorem}\label{th:CF-mm}
If $A\in\mathbb{R}^{n\times n}$ is symmetric, then
$$\lambda_k(A)=\max_{dim(S)=k}\min_{x\in S,\|x\|_2=1}x^TAx,~~k=1,\dots,n.$$
\end{theorem}
\begin{theorem}\label{th:bo:P_DA}
Let $P_D$ be defined by (\ref{de:P_D}) and $\theta$ be the acute angle between $\mathcal{V}$ and $\mathcal{Z}$. Then $P_DA$ has $r$ zero eigenvalues, moreover, if $\cos\theta\ne0$, then the nonzero eigenvalues of $P_DA$ satisfy
\begin{eqnarray*}
\lambda_{min}(\Lambda_\perp)-\epsilon_D\le\lambda(P_DA)\le\lambda_{max}(\Lambda_\perp)+\eta_D,
\end{eqnarray*}
where $\eta_D=\lambda_{max}(\Lambda_\perp)(\sin\theta+\sin^2\theta)$ and
$\epsilon_D = \eta_D+\|E^{-1}\|_2(\|E\|_2+\lambda_{max}(\Lambda_\perp))^2\tan^2\theta$.
\end{theorem}\begin{proof}
It follows from $P_DAZ=0$ that $P_DA$ has $r$ zero eigenvalues. From Theorem \ref{th:tP_D}, we have
\begin{eqnarray}\label{eq:xtP_DAx}
x^T\tilde P_DAx &=& x^TV_\perp\Lambda_\perp V^T_\perp x.
\end{eqnarray}

For all unit 2-norm $x\in\mathbb{R}^n$, it can be written as $x=x_1+x_2$, where $x_1\in\mathcal{Z}$ and $x_2\in\mathcal{Z}^\perp$. Moreover, there exist $t\in\mathbb{R}^r$ and $s\in\mathbb{R}^{n-r}$ such that $x_1=Zt$ and $x_2=Z_\perp s$.
\begin{eqnarray*}
x^TAZE^{-1}Z^TAx &=& x^T_1AZE^{-1}Z^TAx_1+x^T_1AZE^{-1}Z^TAx_2\\
                 & & +x^T_2AZE^{-1}Z^TAx_1+x^T_2AZE^{-1}Z^TAx_2\\
                 &=& t^TZ^TAZt+t^TZ^TAx_2+x^T_2AZt\\
                 & & +x^T_2AZE^{-1}Z^TAx_2\\
                 &=& x^T_1Ax_1+x^T_1Ax_2+x^T_2Ax_1+x^T_2AZE^{-1}Z^TAx_2\\
                 &=& x^TAx-x^T_2Ax_2+x^T_2AZE^{-1}Z^TAx_2.
\end{eqnarray*}
Then we obtain
\begin{eqnarray}\label{eq:xP_DAx}
x^TP_DAx &=& x^TAx-x^TAZE^{-1}Z^TAx\\
         &=& x^T_2Ax_2-x^T_2AZE^{-1}Z^TAx_2.\nonumber
\end{eqnarray}
Substract (\ref{eq:xtP_DAx}) from (\ref{eq:xP_DAx}) on both sides, then we obtain
\begin{eqnarray}\label{eq:P_DA:eigenvalue}
x^TP_DAx = x^T\tilde P_DAx+x^T_2Ax_2-x^TV_\perp\Lambda_\perp V^T_\perp x-x^T_2AZE^{-1}Z^TAx_2.
\end{eqnarray}
\begin{eqnarray*}
x^T_2Ax_2-x^TV_\perp\Lambda_\perp V^T_\perp x &=& x^T_2Ax_2-(x^T_1+x^T_2)V_\perp\Lambda_\perp V^T_\perp(x_1+x_2)\\
                                         &=& x^T_2V\Lambda V^Tx_2-x^T_1V_\perp\Lambda_\perp V^T_\perp x_1\\
                                         & & -x^T_1V_\perp\Lambda_\perp V^T_\perp x_2-x^T_2V_\perp\Lambda_\perp V^T_\perp x_1\\
                                         &=& s^T(Z^T_\perp V)\Lambda(V^TZ_\perp)s-t^T(Z^TV_\perp)\Lambda_\perp(V^T_\perp Z)t\\
                                         & & -t^T(Z^TV_\perp)\Lambda_\perp(V^T_\perp Z_\perp)s-s^T(Z^T_\perp V_\perp)\Lambda_\perp(V^T_\perp Z)t.
\end{eqnarray*}
$\|x_1\|_2=\|t\|_2$ and $\|x_2\|_2=\|s\|_2$ since $Z$ and $Z_\perp$ are orthogonal. Using Lemma \ref{le:angle}, we obtain
\begin{eqnarray}\label{bo:x2x}
|x^T_2Ax_2-x^TV_\perp\Lambda_\perp V^T_\perp x| &\le& (\|x_2\|^2_2\|\Lambda\|_2+\|x_1\|^2_2\|\Lambda_\perp\|_2)\sin^2\theta\\
&   & +\|x_1\|_2\|x_2\|_2\|\Lambda_\perp\|_2\sin2\theta\nonumber.
\end{eqnarray}

For the last item on the right side of (\ref{eq:P_DA:eigenvalue}), we only need to estimate the bound of $Z^T_\perp AZ$ because of $x^T_2AZE^{-1}Z^TAx_2=s^T(Z^T_\perp AZ)E^{-1}(Z^T_\perp AZ)^Ts$. It follows from $(Z,Z_\perp)(Z,Z_\perp)^T=I$ that
$$AZ=Z(Z^TAZ)+Z_\perp(Z^T_\perp AZ).$$
Multiply $V^T_\perp$ from left on both sides of above equation, then we have
\begin{eqnarray*}
(V^T_\perp Z_\perp)(Z^T_\perp AZ) &=& V^T_\perp AZ-(V^T_\perp Z)(Z^TAZ)\\
                                  &=& \Lambda_\perp(V^T_\perp Z)-(V^T_\perp Z)(Z^TAZ).
\end{eqnarray*}
$V^T_\perp Z_\perp$ is invertible since $\cos\theta\ne0$. Thus
\begin{eqnarray}\label{eq:Z_perpAZ}
(Z^T_\perp AZ)=(V^T_\perp Z_\perp)^{-1}\Lambda_\perp(V^T_\perp Z)-(V^T_\perp Z_\perp)^{-1}(V^T_\perp Z)E.
\end{eqnarray}
Using Lemma \ref{le:angle}, we have $\|Z^T_\perp AZ\|_2\le(\|E\|_2+\|\Lambda_\perp\|_2)\tan\theta.$
Thus
\begin{eqnarray}\label{bo:upper:x2x2}
x^T_2AZE^{-1}Z^TAx_2\le\|x_2\|^2_2\|E^{-1}\|_2(\|E\|_2+\|\Lambda_\perp\|_2)^2\tan^2\theta.
\end{eqnarray}
$E^{-1}$ is SPD since $A$ is SPD. Consequently,
\begin{eqnarray}\label{bo:lower:x2x2}
x^T_2AZE^{-1}Z^TAx_2=(Z^TAx_2)^TE^{-1}(Z^TAx_2)\ge0.
\end{eqnarray}
 
$P_DA$ is symmetric since $A$ is SPD. Applying Theorem \ref{th:CF-mm} to \ref{eq:P_DA:eigenvalue} with (\ref{bo:x2x}) and (\ref{bo:lower:x2x2}), we have
\begin{eqnarray*}
\lambda(P_DA) &\le& \lambda(\tilde P_DA)+|x^T_2Ax_2-x^TV_\perp\Lambda_\perp V^T_\perp x|\\
&\le& \lambda_{max}(\Lambda_\perp)+\lambda_{max}(\Lambda_\perp)(\sin\theta+\sin^2\theta).
\end{eqnarray*}
Note that $P_DA$ has $r$ zero eigenvalues. Applying Theorem \ref{th:CF-mm} to \ref{eq:P_DA:eigenvalue} with (\ref{bo:x2x}) and (\ref{bo:upper:x2x2}), the lower bound of the nonzero eigenvalues of $P_DA$ is given by
\begin{eqnarray*}
\lambda(P_DA) &\ge& \lambda(\tilde P_DA)-|x^T_2Ax_2-x^TV_\perp\Lambda_\perp V^T_\perp x|-x^T_2AZE^{-1}Z^TAx_2\\
&\ge& \lambda_{min}(\Lambda_\perp)-\lambda_{max}(\Lambda_\perp)(\sin\theta+\sin^2\theta)\\
& & -\|E^{-1}\|_2(\|E\|_2+\lambda_{max}(\Lambda_\perp))^2\tan^2\theta.
\end{eqnarray*}
Thus the theorem holds.
\end{proof}

Theorem \ref{th:bo:P_DA} shows that in exact arithmetic, as $\theta$ approaches to zero, the maximal and minimal nonzero eigenvalues of $P_DA$ converge to $\lambda_{max}(\Lambda_\perp)$ and $\lambda_{min}(\Lambda_\perp)$ respectively. Hence with an appropriate coarse grid space the spectrum of $P_DA$ will be similar to that of $\tilde P_DA$. Considering rounding error, however, $P_DAZ$ may be not equal to zero matrix, which implies $P_DA$ possibly has some eigenvalues around zero that should be equal to zero in exact arithmetic. So there is a risk for $P_D$ to lead to a poor spectrum of the preconditioned system. We just mention rounding error here since roundoff analysis is outside the scope of this paper.

The authors of \cite{TangNVE} has shown that $P_DM^{-1}A$ can be related to $P_{ADEF1}A$ in terms of their spectrum, which is described in the following theorem.
\begin{theorem}\label{th:relation}
Assume $M$ is an SPD matrix. Let the spectrum of $P_DM^{-1}A$ be given by $\{0,\dots,0,\gamma_{r+1},\dots,\gamma_n\}$ with $\gamma_{r+1}\le\gamma_{r+2}\le\cdots\le\gamma_n$. Let the spectrum of $(M^{-1}P_D+ZE^{-1}Z^T)A$ be $\{1,\dots,1,\mu_{r+1},\dots,\mu_n\}$ with $\mu_{r+1}\le\mu_{r+2}\le\cdots\le\mu_n$. Then, $\gamma_i=\mu_i$ for all $i=r+1,\dots,n$.
\end{theorem}
\begin{proof}
The proof can be found in \cite[Theorem 3.3]{TangNVE}.
\end{proof}
\begin{corollary}\label{co:relation}
Let the spectrum of $P_DA$ be given by $\{0,\dots,0,\lambda_{r+1},\dots,\lambda_n\}$ with $\lambda_{r+1}\le\lambda_{r+2}\le\cdots\le\lambda_n$. Then the spectrum of $P_AA$ is $\{1,\dots,1,\lambda_{r+1},\dots,\lambda_n\}$.
\end{corollary}
\begin{proof}
Let $M=I$. The result follows from Theorem \ref{th:relation}.
\end{proof}
\begin{theorem}
Let $P_A$ be defined by (\ref{de:P_A-DEF1}). Let $\theta$ the acute angle between subspaces $\mathcal{Z}$ and $\mathcal{V}$. If $\cos\theta\ne0$, then the eigenvalues of $P_AA$ satisfy $$\min\{1,\lambda_{min}(\Lambda_\perp)-\epsilon_D\}\le\lambda(P_AA)\le\max\{1,\lambda_{max}(\Lambda_\perp)+\eta_D\},$$
where $\eta_D$ and $\epsilon_D$ are defined in Theorem \ref{th:bo:P_DA}.
\end{theorem}
\begin{proof}
It follows from the combination of Theorem \ref{th:bo:P_DA} and Corollary \ref{co:relation}.
\end{proof}

$P_CA$ is not necessarily symmetric although both $A$ and $P_C$ are symmetric. So we can not apply Courant-Fischer Minimax theorem to estimate the eigenvalue of $P_CA$. Fortunately, the eigenvalues of $P_CA$ are positive since $P_CA$ is similar to a SPD matrix. Thus the eigenvalues of $P_CA$ can be bounded by $\max\{x^TP_CAx\}$ and $\min\{x^TP_CAx\}$ for all unit 2-norm $x\in\mathbb{R}^n$.
\begin{theorem}\label{th:bo:P_CA}
Let $P_C$ be defined by (\ref{de:P_C}). Let $\theta$ be the acute angle between $\mathcal{V}$ and $\mathcal{Z}$. If $\cos\theta\ne0$, then
\begin{eqnarray*}
\lambda_{max}(P_CA) &\le& \max\left\{1+\lambda_{max}(\Lambda), \lambda_{max}(\Lambda_\perp)\right\}+\epsilon_C,\\
\lambda_{min}(P_CA) &\ge& \min\left\{1+\lambda_{min}(\Lambda), \lambda_{min}(\Lambda_\perp)\right\}-\epsilon_C,
\end{eqnarray*}
where $\epsilon_C=\frac{1}{2}(\lambda_{max}(\Lambda_\perp)\|E^{-1}\|_2+1)\tan\theta+\sin\theta+\sin^2\theta.$
\end{theorem}
\begin{proof}
$P_CA$ is similar to $A+A^{1/2}ZE^{-1}Z^TA^{1/2}$ since $A$ is SPD. Moreover $A+A^{1/2}ZE^{-1}Z^TA^{1/2}$ is SPD as well since $A$ and $E$ are SPD. Thus the eigenvalues of $P_CA$ are positive.

For all unit 2-norm $x\in\mathbb{R}^n$, it can be written as $x=x_1+x_2$, where $x_1\in\mathcal{Z}$ and $x_2\in\mathcal{Z_\perp}$. Moreover, there exist $t\in\mathbb{R}^r$ and $s\in\mathbb{R}^{n-r}$ such that $x_1=Zt$ and
$x_2=Z_\perp s$.
\begin{eqnarray}\label{eq:xP_CAx}
x^TP_CAx &=& x^TAx+(x_1+x_2)^TZE^{-1}Z^TA(x_1+x_2)\\
         &=& x^TAx+x^T_1ZE^{-1}Z^TAx_1+x^T_1ZE^{-1}Z^TAx_2\nonumber\\
         &=& x^TAx+x^T_1x_1+x^T_1ZE^{-1}Z^TAx_2.\nonumber
\end{eqnarray}
From the definition of $\tilde P_C$ in (\ref{de:tP_C}), we obtain
\begin{eqnarray}\label{eq:xtP_CAx}
x^T\tilde P_CAx &=& x^TAx+x^TV\tilde E^{-1}V^TAx\\
                &=& x^TAx+x^TVV^Tx.\nonumber
\end{eqnarray}
Subtract (\ref{eq:xtP_CAx}) from (\ref{eq:xP_CAx}) on both sides, then we have
\begin{eqnarray}\label{eq:P_CA:eigenvalue}
x^TP_CAx-x^T\tilde P_CAx &=& x^T_1ZE^{-1}Z^TAx_2+x^T_1x_1-x^TVV^Tx\\
                         &=& x^T_1ZE^{-1}Z^TAx_2+x^T_1V_\perp V^T_\perp x_1\nonumber\\
                         & & -x^T_2VV^Tx_2-2x^T_1VV^Tx_2.\nonumber
\end{eqnarray}
Since both $A$ and $E^{-1}$ are symmetric,
$$x^T_1ZE^{-1}Z^TAx_2=x^T_2AZE^{-1}Z^Tx_1=s^T_2Z^T_\perp AZE^{-1}t.$$
Using (\ref{eq:Z_perpAZ}), we have
\begin{eqnarray*}
(Z^T_\perp AZ)E^{-1}=(V^T_\perp Z_\perp)^{-1}\Lambda_\perp(V^T_\perp Z)E^{-1}-(V^T_\perp Z_\perp)^{-1}(V^T_\perp Z).
\end{eqnarray*}
Note that $\|x_1\|_2=\|t\|_2$ and $\|x_2\|_2=\|s\|_2$. Using Lemma \ref{le:angle}, we obtain
\begin{eqnarray}\label{eq:P_C:ZEZ_TA}
|x^T_1ZE^{-1}Z^TAx_2|\le\|x_1\|_2\|x_2\|_2(\|\Lambda_\perp\|_2\|E^{-1}\|_2+1)\tan\theta.
\end{eqnarray}
Since $\|x\|_2=1$, we have the following bounds with Lemma \ref{le:angle}
\begin{eqnarray}\label{ineq:vbounds}
0\le x^T_1V_\perp V^T_\perp x_1\le\|x_1\|^2_2\sin^2\theta,\nonumber\\
0\le x^T_2VV^Tx_2\le\|x_2\|^2_2\sin^2\theta,\\
|x^T_1VV^Tx_2|\le\|x_1\|_2\|x_2\|_2\sin\theta.\nonumber
\end{eqnarray}
With (\ref{eq:P_CA:eigenvalue}), (\ref{eq:P_C:ZEZ_TA}) and (\ref{ineq:vbounds}), we obtain
\begin{eqnarray*}
\lambda_{min}(\tilde P_CA)-\epsilon_C \le x^TP_CAx \le \lambda_{max}(\tilde P_CA)+\epsilon_C,
\end{eqnarray*}
where $\epsilon_C=\frac{1}{2}(\lambda_{max}(\Lambda_\perp)\|E^{-1}\|_2+1)\tan\theta+\sin\theta+\sin^2\theta.$ 
Thus the theorem follows from $\min\{x^TP_CAx\}\le\lambda(P_CA)\le\max\{x^TP_CAx\}$ for all unit 2-norm $x\in\mathbb{R}^n$.
\end{proof}

Theorem \ref{th:bo:P_CA} shows that the maximal and minimal eigenvalues of $P_CA$ converge to $\max\{\lambda_{max}(\Lambda_\perp), 1+\lambda_{max}(\Lambda)\}$ and $\min\{\lambda_{min}(\Lambda_\perp), 1+\lambda_{min}(\Lambda)\}$ respectively as $\theta$ approaches to zero. Therefore it can be expected that with an appropriate coarse grid space $P_CA$ will have the similar spectrum of $\tilde P_CA$. Analogously, it can be proved that the maximal and minimal eigenvalues of $AP_C$ have the same bounds as that of $P_CA$.

\section{Inexact inverse of projection matrix}For a large projection matrix, it is expensive to solve its inverse exactly. So we want to replace a projection matrix by an approximate one that is cheap to compute its inverse. In this seciton, we assume $\tilde H$ is an invertible matrix as an approximation to $\tilde E$ in (\ref{de:tP_D}).
\begin{theorem}\label{th:bo:tP_D:inexactE}
Let $\bar P_D=I-AV\tilde H^{-1}V^T$ and $\rho=\tilde E\tilde H^{-1}-I$. If the eigenvalues of $\bar P_DA$ are real, then
\begin{eqnarray*}
-\xi_D\le\lambda(\bar P_DA)\le\lambda_{max}(\Lambda_\perp)+\xi_D,
\end{eqnarray*}
where $\xi_D=\|\rho\|_2\|\Lambda\|_2.$
\end{theorem}
\begin{proof}
$\|\rho\|_2$ measures the distance between $\tilde H$ and $\tilde E$ when $\tilde H$ is used as the right inverse of $\tilde E$. Obviously $\|\rho\|_2=0$ if and only if $\tilde E=\tilde H$. Assume $\|x\|_2=1$ and use the definition of $\tilde P_D$ in (\ref{de:tP_D}), then we obtain
\begin{eqnarray*}
x^T\bar P_DAx &=& x^TAx-x^TAV\tilde H^{-1}V^TAx\\
              &=& x^TAx-x^TVV^TAx-x^TV\rho V^TAx\nonumber\\
              &=& x^T\tilde P_DAx-x^TV\rho \Lambda V^Tx.\nonumber
\end{eqnarray*}
Since $|x^TV\rho\Lambda V^Tx|\le\|\rho\|_2\|\Lambda\|_2$,
$$-\|\rho\|_2\|\Lambda\|_2\le x^T\bar P_DAx \le\lambda_{max}(\tilde P_DA)+\|\rho\|_2\|\Lambda\|_2.$$
Therefore the theorem follows from $\min\{x^T\bar P_DAx\}\le\lambda(\bar P_DA)\le\max\{x^T\bar P_DAx\}$.
\end{proof}

Theorem \ref{th:bo:tP_D:inexactE} shows that if the eigenvalues of $\bar P_DA$ are real, the maximal and minimal eigenvalues of $\bar P_DA$ converge to $\lambda_{max}(\Lambda_\perp)$ and $0$ respectively as $\|\rho\|_2$ approaches zero. That implies that $\bar P_DA$ will have small eigenvalues around zero. As a result, the spectrum of $\bar P_DA$ will become worse than that of $\tilde P_DA$.
\begin{theorem}\label{th:bo:tP_C:inexactE}
Let $\bar P_C=I+V\tilde H^{-1}V^T$ and $\rho=\tilde H^{-1}\tilde E-I$. If the eigenvalues of $\bar P_CA$ are real, then
\begin{eqnarray*}
\lambda_{max}(\bar P_CA) &\le& \max\left\{1+\lambda_{max}(\Lambda), \lambda_{max}(\Lambda_\perp)\right\}+\xi_C,\\
\lambda_{min}(\bar P_CA) &\ge& \min\left\{1+\lambda_{min}(\Lambda), \lambda_{min}(\Lambda_\perp)\right\}-\xi_C,
\end{eqnarray*}
where $\xi_C=\|\rho\|_2$.
\end{theorem}
\begin{proof}
$\|\rho\|_2$ measures the distance between $\tilde H$ and $\tilde E$ when $\tilde H$ is used as the left inverse of $\tilde E$. Obviously $\|\rho\|_2=0$ if and only if $\tilde E=\tilde H$. Assume $\|x\|_2=1$ and use the definition of $\tilde P_C$ in (\ref{de:tP_C}), then we have
\begin{eqnarray*}
x^T\bar P_CAx &=& x^TAx+x^TV\tilde H^{-1}V^TAx\\
              &=& x^TAx+x^TV\tilde E^{-1}V^TAx+x^TV\rho V^Tx\nonumber\\
              &=& x^T\tilde P_CAx+x^TV\rho V^Tx.\nonumber
\end{eqnarray*}
Since $|x^TV\rho V^Tx|\le\|\rho\|_2$,
$$\lambda_{min}(\tilde P_CA)-\|\rho\|_2\le x^T\bar P_CAx \le\lambda_{max}(\tilde P_CA)+\|\rho\|_2.$$
Therefore the theorem follows from $\min\{x^T\bar P_CAx\}\le\lambda(\bar P_CA)\le\max\{x^T\bar P_CAx\}$.
\end{proof}

Theorem \ref{th:bo:tP_C:inexactE} shows that if the eigenvalues of $\bar P_CA$ are real, as $\rho$ approaches to zero, the maximal and minimal eigenvalues of $\bar P_CA$ converge to $\max\{\lambda_{max}(\Lambda_\perp), 1+\lambda_{max}(\Lambda)\}$ and $\min\{\lambda_{min}(\Lambda_\perp), 1+\lambda_{min}(\Lambda)\}$ respectively. That implies $\bar P_CA$ and $\tilde P_CA$ have the similar eigenvalue distribution if $\|\rho\|_2$ is small enough.
\begin{theorem}\label{th:bo:tP_A:inexactE}
Let $\bar P_A=I-AV\tilde H^{-1}V^T+V\tilde H^{-1}V^T$. Let $\rho_1=\tilde E\tilde H^{-1}-I$ and $\rho_2=\tilde H^{-1}\tilde E-I$. If the eigenvalues of $\bar P_AA$ are real, then
\begin{eqnarray*}
\lambda_{max}(\bar P_AA) &\le& \max\left\{1, \lambda_{max}(\Lambda_\perp)\right\}+\xi_A,\\
\lambda_{min}(\bar P_AA) &\ge& \min\left\{1, \lambda_{min}(\Lambda_\perp)\right\}-\xi_A,
\end{eqnarray*}
where $\xi_A=\|\rho_1\|_2\|\Lambda\|_2+\|\rho_2\|_2$.
\end{theorem}
\begin{proof}
Assume $\|x\|_2=1$ and use the definition of $\tilde P_A$ in (\ref{de:tP_ADEF1}), then we have
\begin{eqnarray*}
x^T\bar P_AAx &=& x^TAx-x^TAV\tilde H^{-1}V^TAx+x^TV\tilde H^{-1}V^TAx\\
              &=& x^TAx-x^TV\rho_1\Lambda V^Tx+x^TV\rho_2V^Tx\\
              & & -x^TAV\tilde E^{-1}V^TAx+x^TV\tilde E^{-1}V^TAx\\
              &=& x^T\tilde P_AAx-x^TV\rho_1\Lambda V^Tx+x^TV\rho_2V^Tx.
\end{eqnarray*}
Since $|x^TV\rho_1\Lambda V^Tx-x^TV\rho_2V^Tx|\le\|\rho_1\|_2\|\Lambda\|_2+\|\rho_2\|_2$, 
$$\lambda_{min}(\tilde P_AA)-\|\rho_1\|_2\|\Lambda\|_2-\|\rho_2\|_2\le x^T\bar P_AAx\le\lambda_{max}(\tilde P_AA)+\|\rho_1\|_2\|\Lambda\|_2+\|\rho_2\|_2 .$$
Thus the theorem follows from $\min\{x^T\bar P_AAx\}\le\lambda(\bar P_AA)\le\max\{x^T\bar P_AAx\}$.
\end{proof}

Similarly, Theorem \ref{th:bo:tP_A:inexactE} shows that the maximal and minimal eigenvalues of $\bar P_AA$ converge to $\max\{1, \lambda_{max}(\Lambda_\perp)\}$ and $\min\{1, \lambda_{min}(\Lambda_\perp)\}$ respectively as both $\|\rho_1\|_2$ and $\|\rho_2\|_2$ approach to zero, if the eigenvalues of $\bar P_AA$ are real. That implies $\bar P_CA$ and $\tilde P_CA$ have the similar eigenvalue distribution if $\tilde H$ is a good approximation to $\tilde E$.

\section{Numerical experiments}
In this section, numerical comparison of various preconditioners will be reported. All tests were performed by using Matlab(R2010b) on an Intel Core2 Duo E7500, 2.93GHz processor with 4Gb memory. Except for the deflation preconditioner, the system preconditioned by the coarse grid correction and adapted deflation preconditioners are not necessarily symmetric although $A$ is SPD. Therefore we apply GMRES\cite{Saad} to solve the preconditioned system iteratively. In addition, we apply Gram-Schmidt method with reorthogonalization to maintain the orthogonality of basis of Krylov subspace\cite{Gene, Saad}.
\subsection{Diagonal matrix}
The first test case is a diagonal matrix with entries $10^{-7}$, $10^{-6}$, $\cdots$, $10^{-1}$, 1, 10, 10.1, 10.2, $\cdots$, 209.1. The right hand side is a vector of all ones. The matrix has 7 small eigenvalues less than 1 to be removed. The eigenvectors associated with these eigenvalues are the unit vectors, i.e., $V=(e_1,e_2,\dots,e_7)$, where $e_i$ is the $i$th column of identity matrix. All tests are required to reduce the relative residual below $10^{-12}$. The initial guess vector is always zero vector. GMRES method converges within 273 iterations without preconditioner. The perturbations in coarse grid space and projection matrix are generated by Matlab function rand.

\begin{table}[htbp!]
\caption{The distance between span\{$V$\} and span\{$V+rand/\epsilon$\}.}
\centering\small
\begin{tabular}{|c|c|c|c|c|c|}\hline
& $\epsilon=1e+01$ & $\epsilon=1e+02$ & $\epsilon=1e+03$ & $\epsilon=1e+04$ & $\epsilon=1e+05$ \\ \hline
$\sin\theta$ & 9.77e-01 & 5.06e-01 & 6.03e-02 & 6.05e-03 & 6.02e-04 \\ \hline
$res_{max}$ & 7.08e+01 & 5.54e+01 & 7.33     & 5.78e-01 & 4.43e-02 \\ \hline
\end{tabular}\label{ta:distance}
\end{table}
Table \ref{ta:distance} shows the distance between the exact coarse grid space and the coarse grid space with various perturbation. It should be noted that it is expensive and unnecessary in practice to compute $\sin\theta$ by Lemma \ref{le:angle} for a general linear system. Assume that $(\tilde\lambda_i, \tilde v_i)(i=1,\cdots,r)$ are Ritz pairs that are extracted from the perturbed coarse grid space by Rayleigh-Ritz procedure\cite{MatrixAlgII}. In general, $\max_{i=1,\cdots,r}\{\|A\tilde v_i-\tilde\lambda_i\tilde v_i\|_2\}$ denoted by $res_{max}$ in Table \ref{ta:distance} decreases as the two subspaces approach to each other. So we can use $res_{max}$ to measure the distance between the two subspaces since it is cheap to obtain.

\begin{table}[htbp!]
\caption{The number of iterations of GMRES with various preconditioners that are constructed with $Z=V+rand/\epsilon$ and $E^{-1}$.}
\centering\small
\begin{tabular}{|c|c|c|c|c|c|c|}\hline
        & $V$, $\tilde E^{-1}$ & $\epsilon=1e+01$ & $\epsilon=1e+02$ & $\epsilon=1e+03$ & $\epsilon=1e+04$ & $\epsilon=1e+05$ \\ \hline
 $P_D$  & 71  & $>$300        & $>$300          & $>$300          & 109             & 88  \\ \hline
 $P_C$  & 104 & 273           & 267             & 222             & 173             & 144 \\ \hline
 $P_A$  & 72  & 290           & 231             & 167             & 117             & 96  \\ \hline
\end{tabular}\label{ta:numberiteration-Z}
\end{table}
\begin{figure}[htbp!]
\begin{minipage}{0.48\linewidth}
\centering
\includegraphics[width=\textwidth]{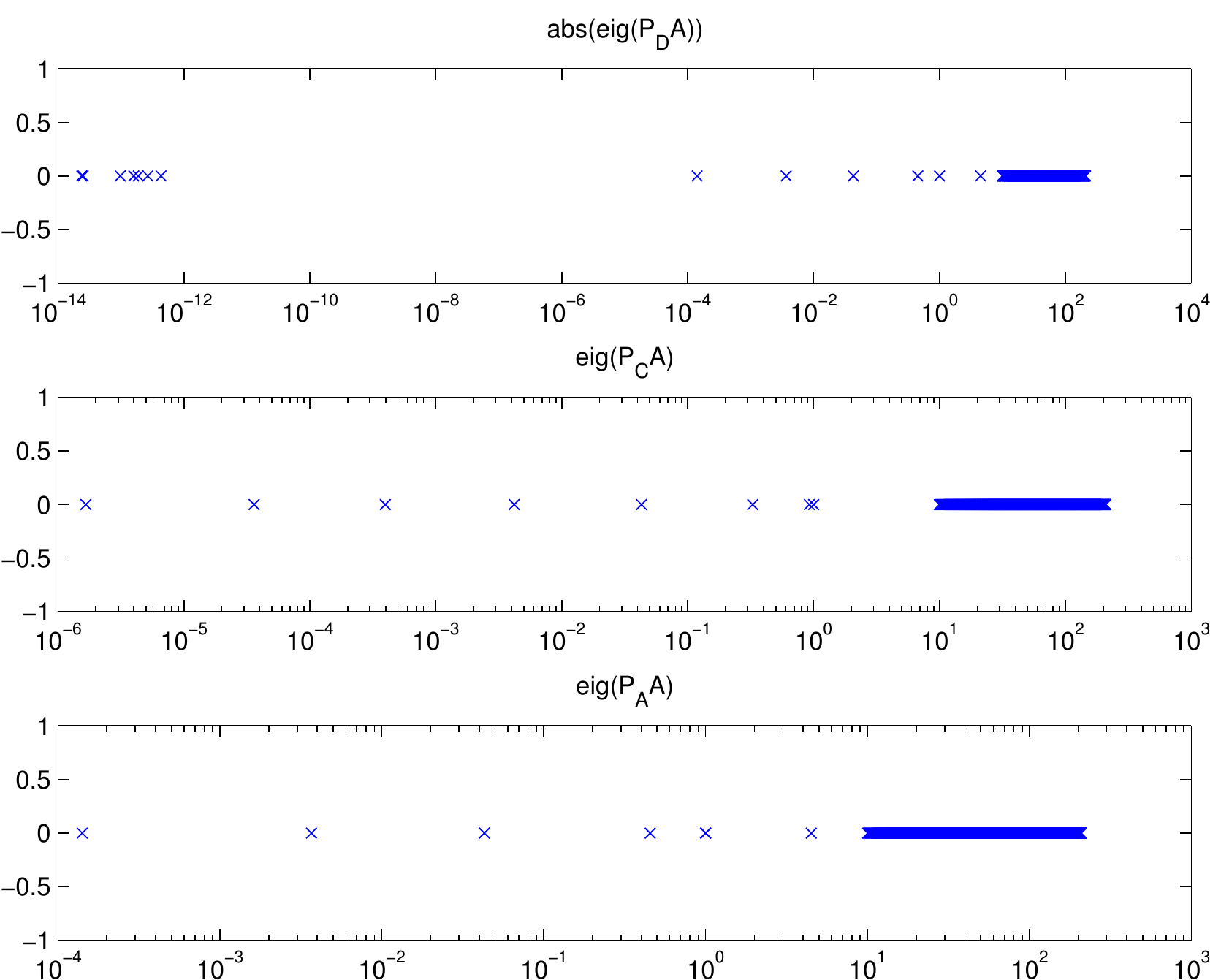}
\caption{The eigenvalue distribution of the system with the three preconditioners constructed with $Z=V+rand/1e+03$ and $E^{-1}$.}
\label{fig:diagonal-spectrum-V1e3}
\end{minipage}
\hfill
\begin{minipage}{0.48\linewidth}
\centering
\includegraphics[width=\textwidth]{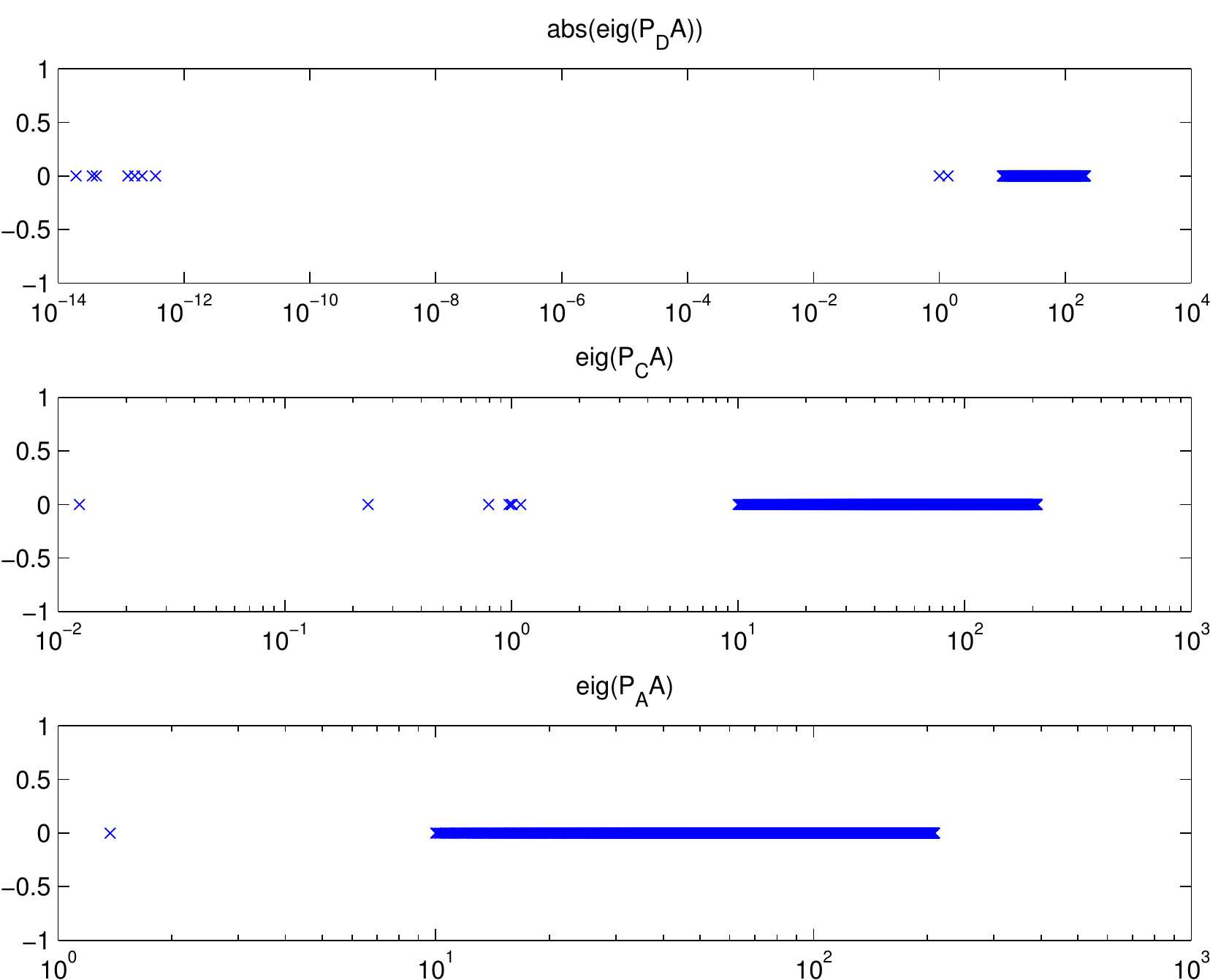}
\caption{The eigenvalue distribution of the system with the three preconditioners constructed with $Z=V+rand/1e+05$ and $E^{-1}$.}
\label{fig:diagonal-spectrum-V1e5}
\end{minipage}
\end{figure}
In Table \ref{ta:numberiteration-Z}, the second column shows the number of iterations needed for convergence for GMRES with the three preconditioners in the case that both coarse grid space and $E^{-1}$ are exact. The columns 3-7 show the number of iterations required for convergence in the case that only coarse grid space is perturbed. As is shown, all preconditioners suffer from the perturbation in the coarse grid space if it is large(see the third column). As the perturbation decreases, $P_C$ and $P_A$ behave better, whereas $P_D$ behaves better only if the perturbation is small enough. On the other hand, if the perturbed coarse grid space is close enough to the exact one(see the last two columns in Table \ref{ta:numberiteration-Z}), $P_D$ is more efficient than $P_A$ and both of them are much more efficient than $P_D$.

Figure \ref{fig:diagonal-spectrum-V1e3} and Figure \ref{fig:diagonal-spectrum-V1e5} show the eigenvalue distribution of the system preconditioned by the three preconditioners. As discussed in previous section, because of rounding error, $P_DA$ has some tiny eigenvalues around zero. For this test case, rounding error does not impact the spectrum of $P_DA$ when the perturbation is small enough.

\begin{table}[htbp!]
\caption{The number of iterations of GMRES with different preconditioners, where only projection matrix is perturbed and the perturbation of $\tilde E$ is $\tilde H=\tilde E+rand/\epsilon$.}
\centering\small
\begin{tabular}{|c|c|c|c|c|}\hline
        & $\epsilon=1e+10$ & $\epsilon=1e+12$ & $\epsilon=1e+14$ & $\epsilon=1e+16$ \\ \hline
 $P_D$  & $>$300        & $>$300          & $>$300          & $>$300    \\ \hline
 $P_C$  & 111           & 104             & 104             & 104    \\ \hline
 $P_A$  & 88            & 88              & 80              & 80     \\ \hline
 $\|\tilde H^{-1}\tilde E-I\|_2$ & 1.68e-03  & 1.08e-05 & 1.77e-07 & 1.23e-09 \\ \hline
 $\|\tilde E\tilde H^{-1}-I\|_2$ & 1.68e-03  & 1.08e-05 & 1.77e-07 & 1.23e-09 \\ \hline
\end{tabular}\label{ta:numberiteration-E}
\end{table}
In Table \ref{ta:numberiteration-E}, the first three rows show the number of iterations required for convergence in the case that only projection matrix is perturbed. The distance between $\tilde H^{-1}$ and $\tilde E^{-1}$ are reported on the last two rows. As is shown, $P_C$ and $P_A$ are robust on the perturbation in projection matrix for all four cases, while $P_D$ is sensitive to that even though $\tilde H^{-1}$ is very close to $\tilde E^{-1}$(see the last column).

\begin{figure}[htbp!]
\begin{minipage}{0.48\linewidth}
\centering
\includegraphics[width=\textwidth]{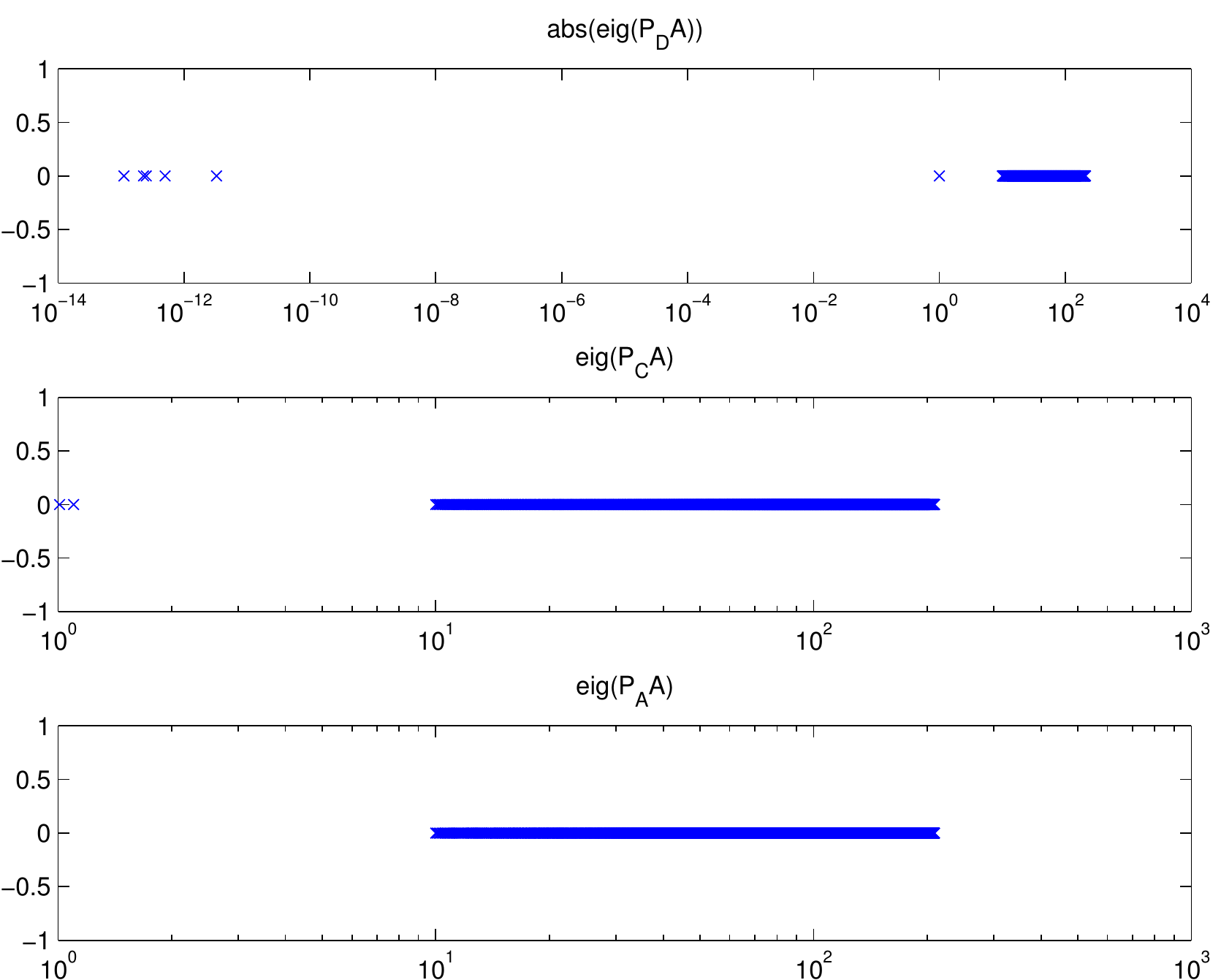}
\caption{The eigenvalue distribution of the system with different preconditioners constructed with $V$ and $\tilde H=\tilde E+rand/1e+12$.}
\label{fig:diagonal-perturbation-E1e12}
\end{minipage}
\hfill
\begin{minipage}{0.48\linewidth}
\centering
\includegraphics[width=\textwidth]{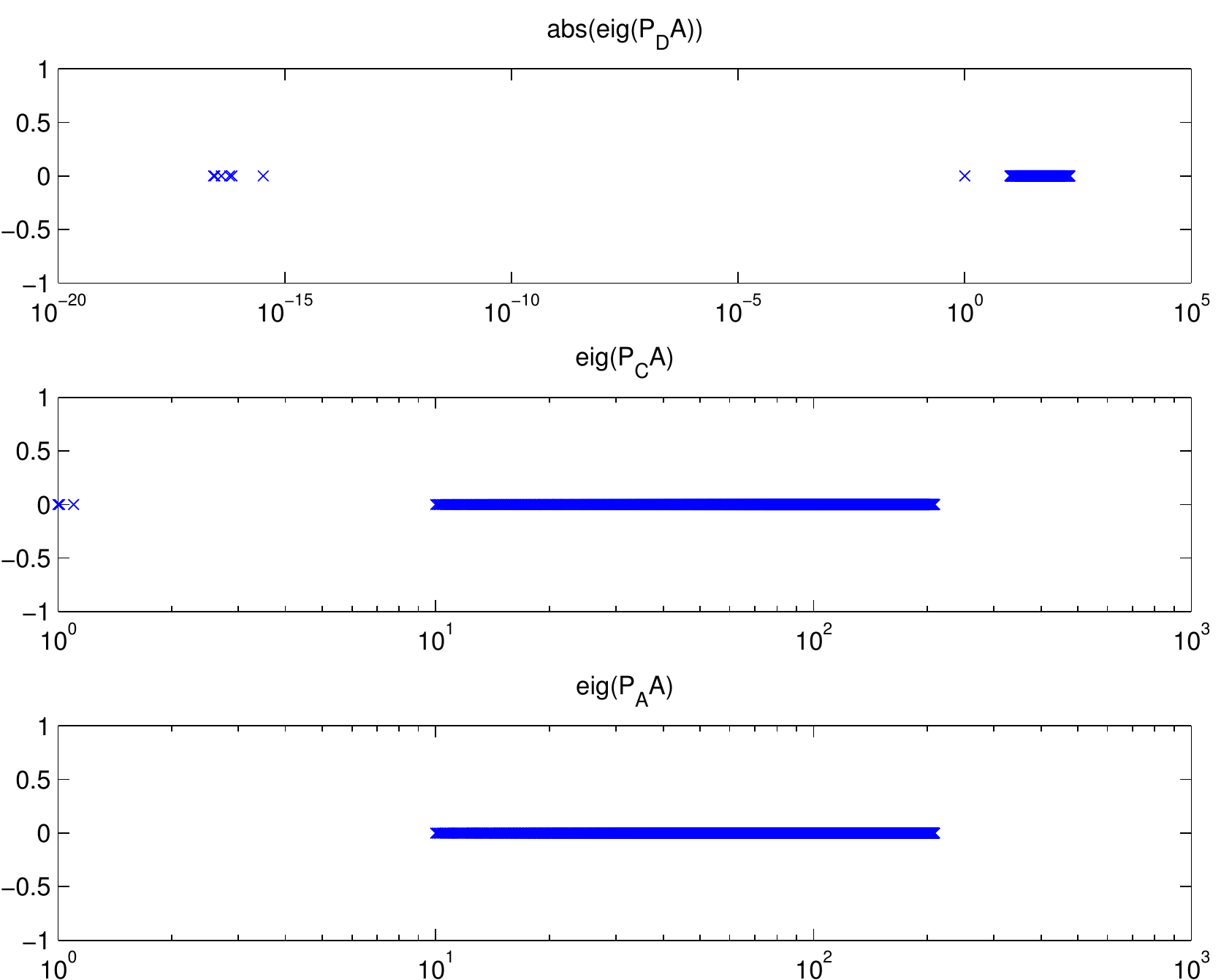}
\caption{The eigenvalue distribution of the system with different preconditioners constructed with $V$ and $\tilde H=\tilde E+rand/1e+16$.}
\label{fig:diagonal-perturbation-E1e16}
\end{minipage}
\end{figure}
Figure \ref{fig:diagonal-perturbation-E1e12} and Figure \ref{fig:diagonal-perturbation-E1e16} show $P_C$ and $P_A$ successfully shift the small eigenvalues of $A$ to the large ones; however, $P_D$ fails to deflate the small eigenvalues and generates much smaller ones, which leads to a worse eigenvalue distribution.

\subsection{Boundary value problem}
We solve the following model problem
\begin{eqnarray*}
-\nabla\cdot(\kappa\nabla u) &=& f~~\mbox{in}~~\Omega=[0,1]^2,\\
u &=& 0~~\mbox{on}~~\partial\Omega,
\end{eqnarray*}
by the two-level multiplicative Schwarz method\cite{NatafMikolajZhao}. The model problem is discretized by FreeFem++\cite{Hecht} and we obtain a coefficient matrix with the order of 10201. Tests are performed on irregular overlapping decompositions with the overlap of 2 elements. These overlap decompositions are built by adding the immediate neighboring vertices to non-overlapping subdomain obtained by Metis\cite{Kumar}.

In the two-level multiplicative Schwarz method, one-level preconditioner is the restricted additive Schwarz preconditioner(RAS)\cite{CaiSarkins} that is responsible to remove high frequency models of the original system, and the deflation, coarse grid correction and adapted deflation preconditioners are applied as two-level preconditioners that are responsible to remove lower frequency ones of the system preconditioned by one-level preconditioner.

We choose Ritz vectors to span the coarse grid space, which are extracted from Krylov subspace during the solve of the system preconditioned by RAS. These vectors are approximations to the eigenvectors corresponding to the lower part of the spectrum of the system preconditioned by RAS. To enrich the information on lower part of the spectrum, we use Ritz vectors in a splitting way to construct the coarse grid space, see \cite{NatafHuaVic, NatafMikolajZhao, TangNVE} and references therein. More precisely, let
\begin{eqnarray*}
V = \left[\begin{array}{cccc}
		v_{11} & v_{12} & \cdots & v_{1,r}   \\
		v_{21} & v_{22} & \cdots & v_{2,r}   \\
		\cdots & & & \\
        v_{nparts,1} & v_{nparts,2} & \cdots & v_{nparts,r}
	\end{array}\right]
\end{eqnarray*}
store Ritz vectors columnwise, where $nparts$ is the number of subdomains and $r$ the number of Ritz vectors; let $Z_i$ store the orthogonal vectors obtained by orthogonalizing $(v_{i1}, v_{i2}, \cdots, v_{i,r})$. Then $Z$ is formed in the following way
\begin{eqnarray*}
Z = \left[\begin{array}{cccc}
		Z_1    & 0      & \cdots & 0      \\
		0      & Z_2    & \cdots & 0      \\
		\vdots & \vdots & \cdots & \vdots \\
        0      & 0      & \cdots & Z_{nparts}
	\end{array}\right].
\end{eqnarray*}

Obviously, span\{$V$\} is a subspace of span\{$Z$\}, moreover span\{$Z$\} is $nparts$ times as large as span\{$V$\}. So we can expect span\{$Z$\} has richer information corresponding to small eigenvalues. Note that $Z$ has a very sparse structure, which leads to a sparse structure of $E$ in (\ref{de:projection}) as well.

Comparison of the three preconditioners are performed on two different configurations with highly heterogeneous viscosity. See \cite{NatafHuaVic} for details. Two cases are described as following:
\begin{itemize}
\item skyscraper viscosity: for $x$ and $y$ such that for [9x]$\equiv$0(mod 2) and [9y]$\equiv$0(mod 2), $\kappa=10^4([9y]+1)$; and $\kappa=1$ elsewhere. See Figure \ref{fig:bvp_skyscraper}.
\item continuous viscosity: $\kappa(x,y)=10^6/3\sin(4\pi(x+y)+0.1)$. See Figure \ref{fig:bvp_continuous}.
\end{itemize}

\begin{figure}[htbp!]
\begin{minipage}[t]{0.48\linewidth}
\centering
\includegraphics[width=\textwidth]{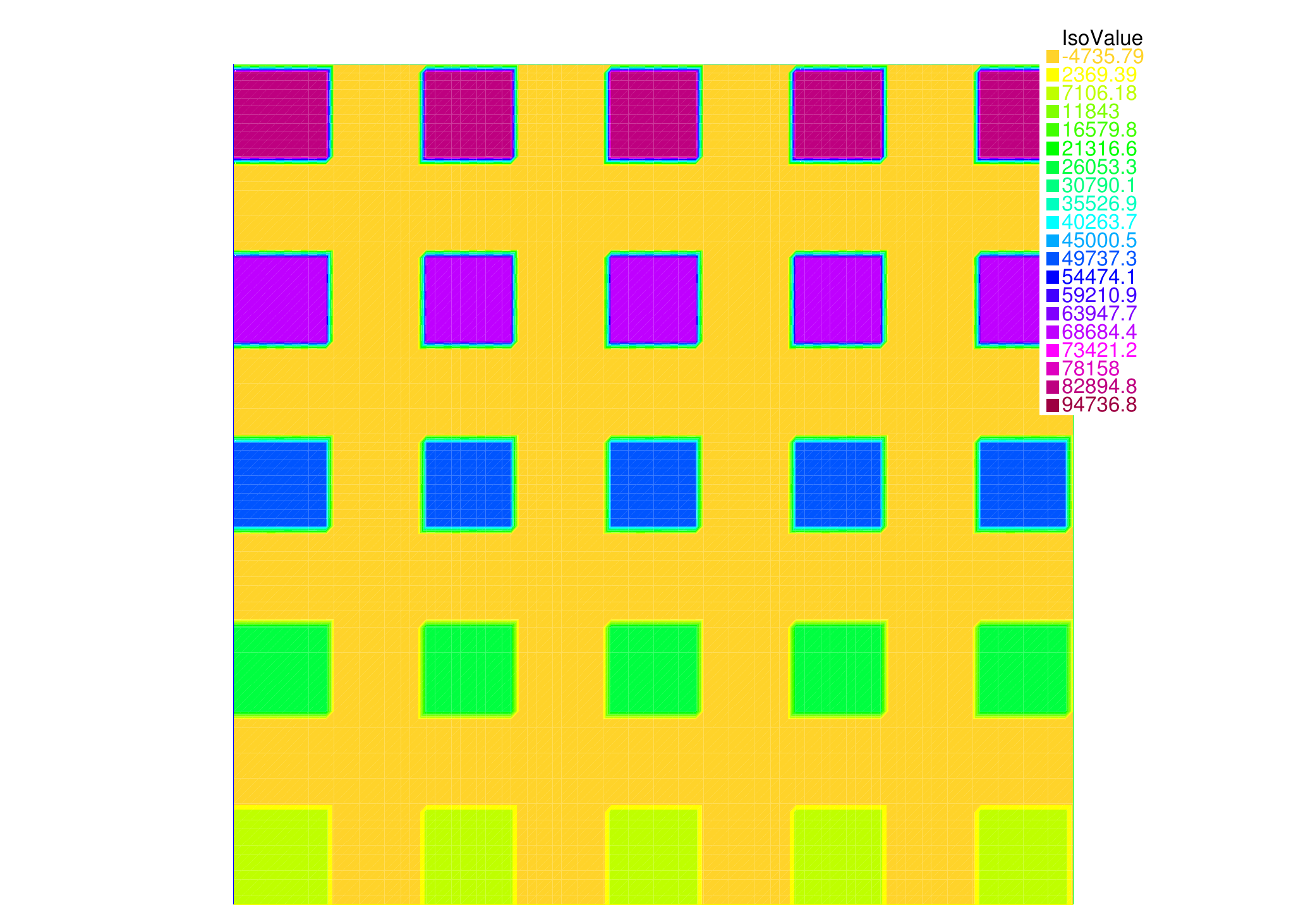}
\caption{Skyscraper case}
\label{fig:bvp_skyscraper}
\end{minipage}
\hfill
\begin{minipage}[t]{0.48\linewidth}
\centering
\includegraphics[width=\textwidth]{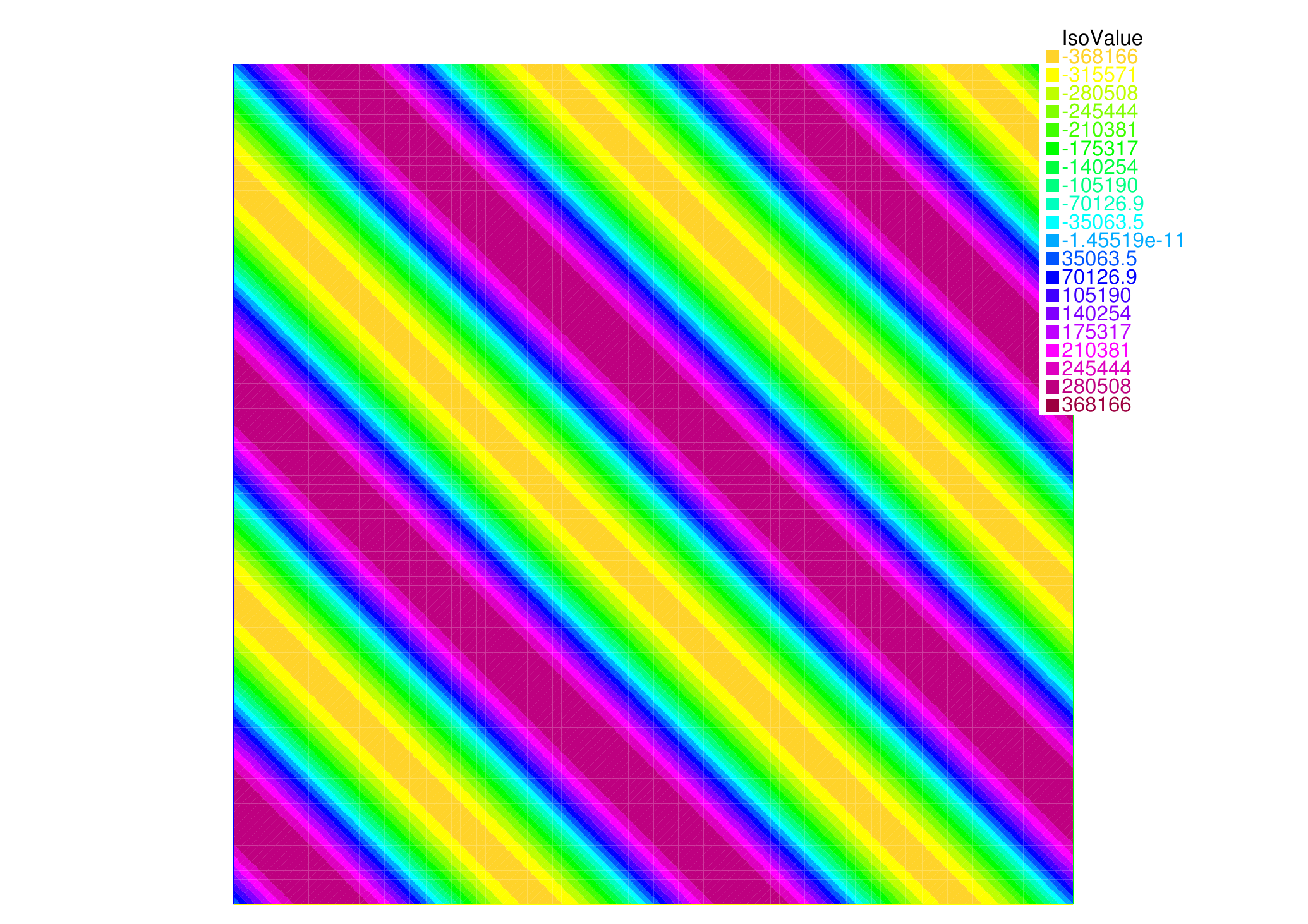}
\caption{Continuous case}
\label{fig:bvp_continuous}
\end{minipage}
\end{figure}

\begin{figure}[htbp!]
\begin{minipage}[t]{0.48\linewidth}
\centering
\includegraphics[width=\textwidth]{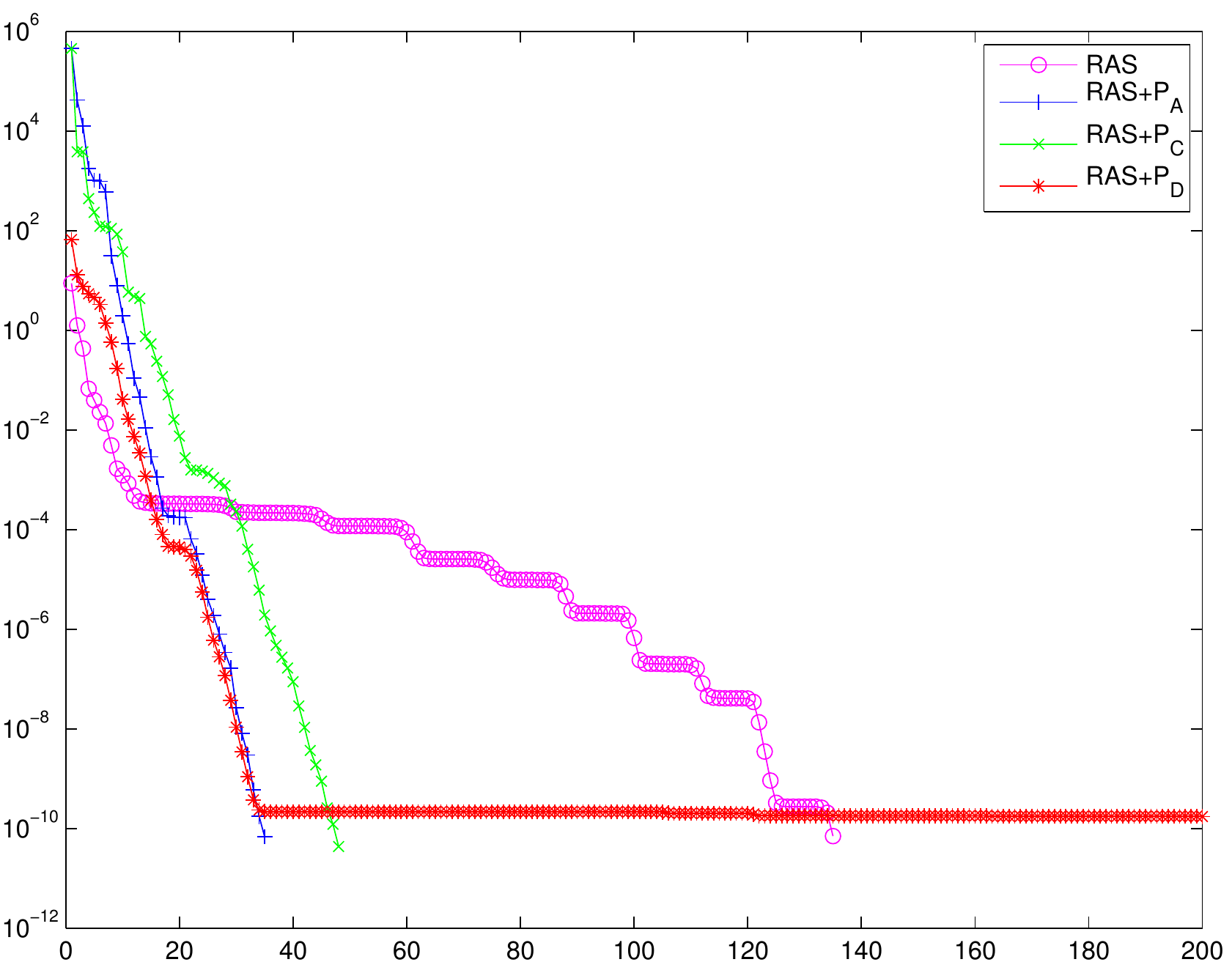}
\caption{Skyscraper case: 16 subdomains, 15 Ritz vectors spanning coarse grid space.}
\label{fig:bvp_s-Ritz14-part16}
\end{minipage}
\hfill
\begin{minipage}[t]{0.48\linewidth}
\centering
\includegraphics[width=\textwidth]{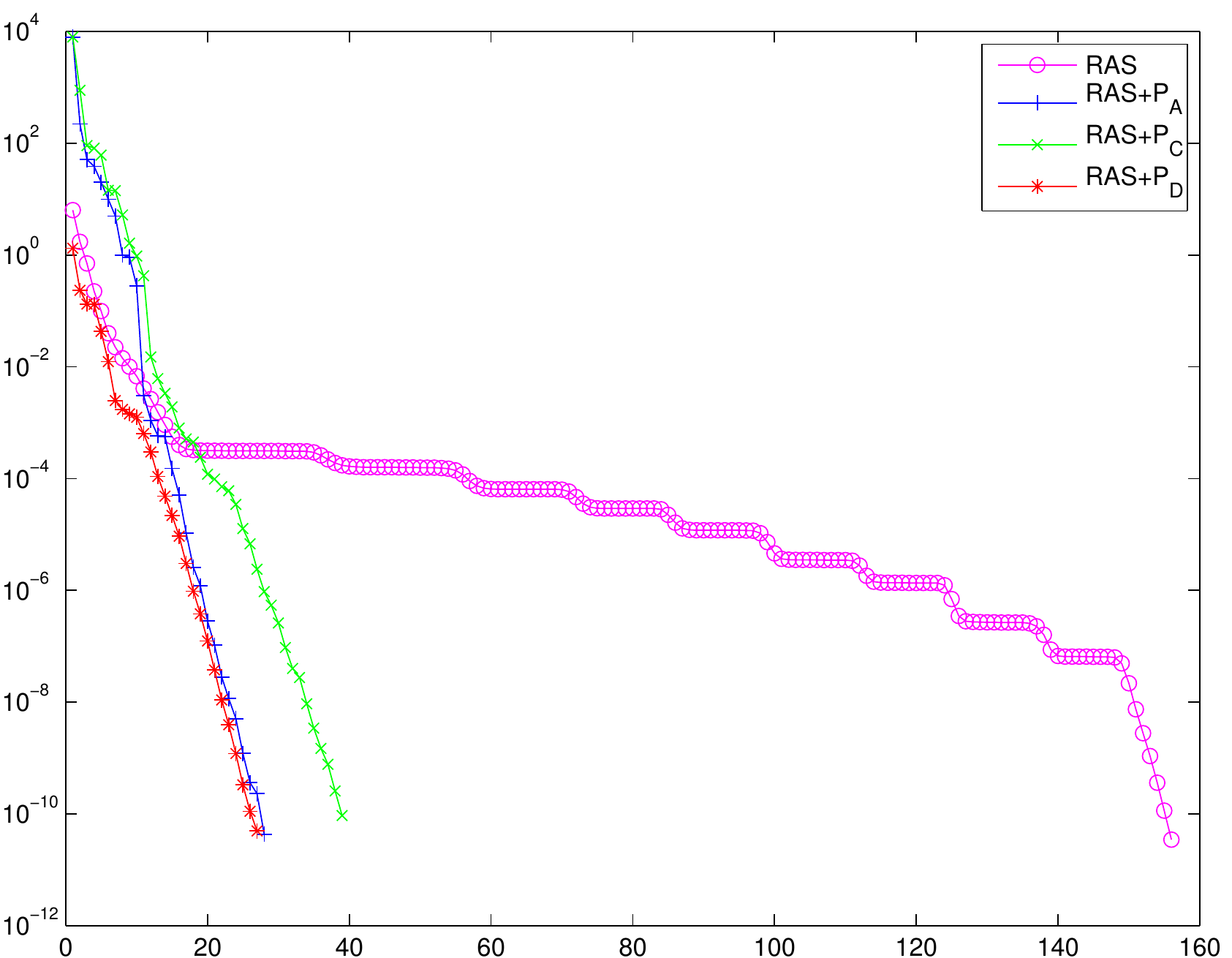}
\caption{Skyscraper case: 32 subdomains, 16 Ritz vectors spanning coarse grid space.}
\label{fig:bvp_s-Ritz15-part32}
\end{minipage}
\end{figure}

\begin{figure}[htbp!]
\begin{minipage}[t]{0.48\linewidth}
\centering
\includegraphics[width=\textwidth]{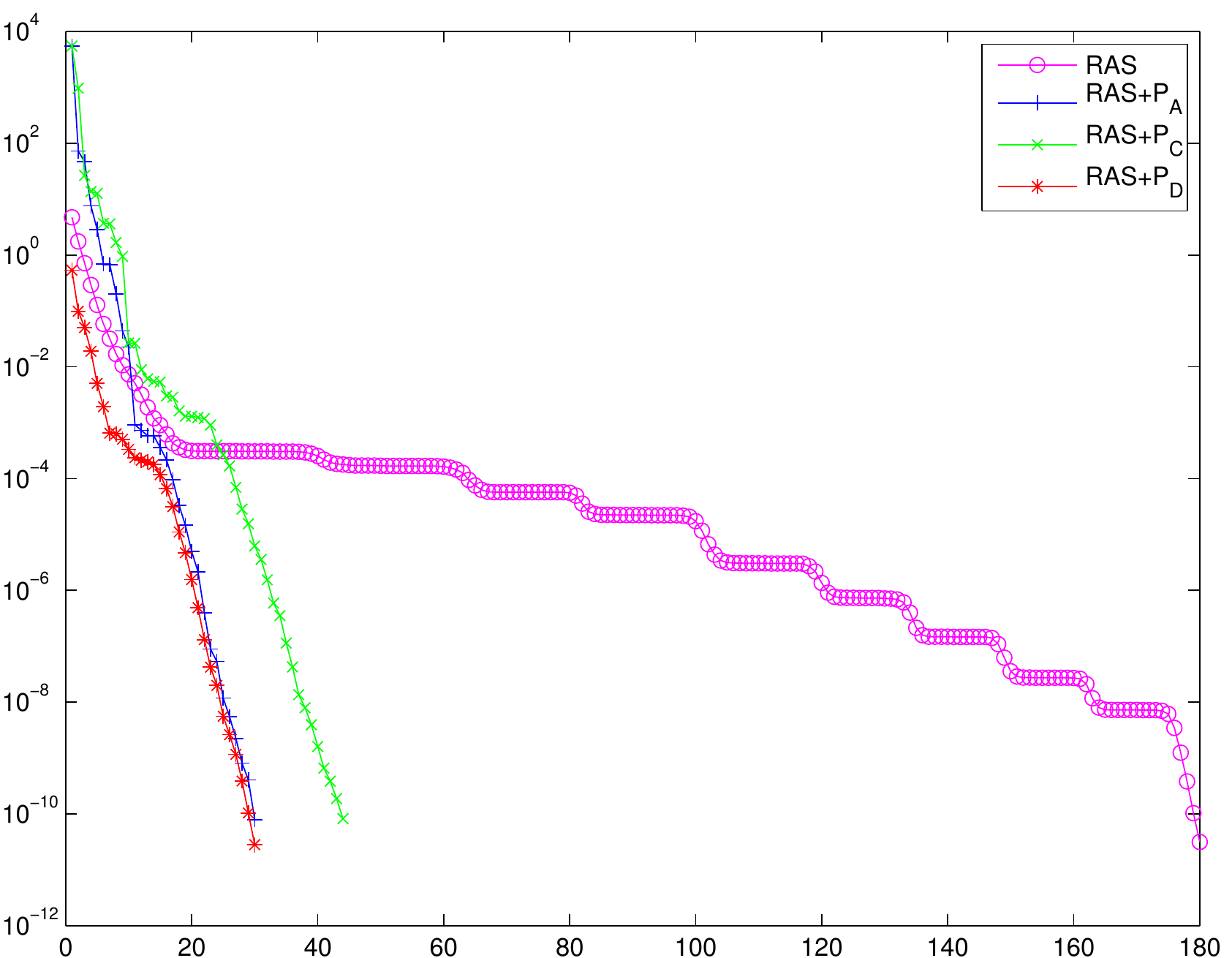}
\caption{Skyscraper case: 64 subdomains, 20 Ritz vectors spanning coarse grid space.}
\label{fig:bvp_s-Ritz19-part64}
\end{minipage}
\hfill
\begin{minipage}[t]{0.48\linewidth}
\centering
\includegraphics[width=\textwidth]{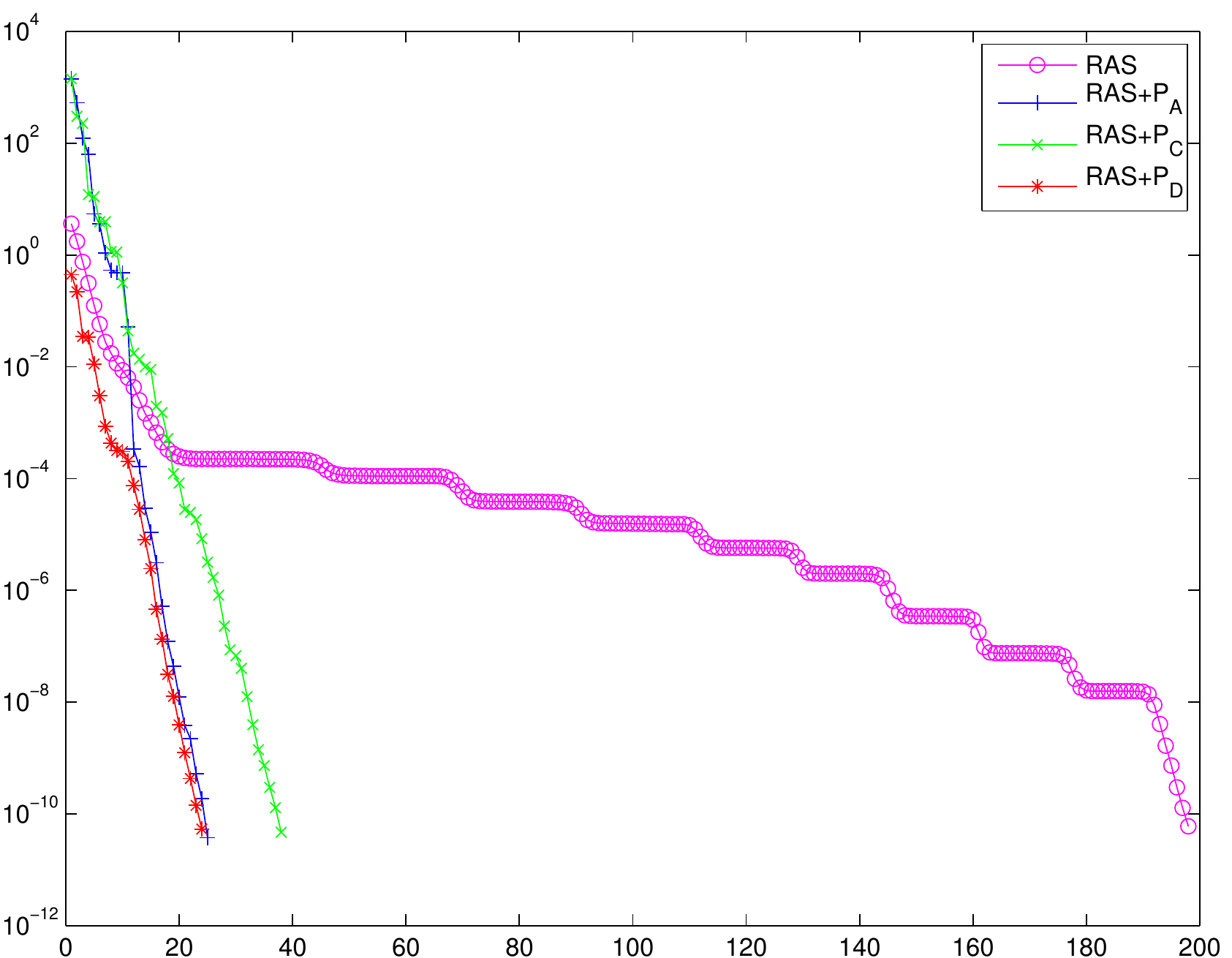}
\caption{Skyscraper case: 128 subdomains, 20 Ritz vectors spanning coarse grid space.}
\label{fig:bvp_s-Ritz20-part128}
\end{minipage}
\end{figure}
Figures \ref{fig:bvp_s-Ritz14-part16}-\ref{fig:bvp_s-Ritz20-part128} plot the convergence curves for the case of skyscraper viscosity against the various number of subdomains. As is shown, $P_D$ has a long plateau in the convergence on the decomposition with 16 subdomains, while $P_A$ and $P_C$ improve convergence significantly. On the other decompositions, $P_D$ has almost the same number of iterations as $P_A$ although the initial residual of $P_D$ is considerably less than $P_A$, moreover $P_D$ and $P_A$ performs better than $P_C$. In addition, two-level method with RAS and $P_A$ and $P_C$ outperforms one-level method with RAS on all four decompositions.

\begin{figure}[htbp!]
\begin{minipage}[t]{0.48\linewidth}
\centering
\includegraphics[width=\textwidth]{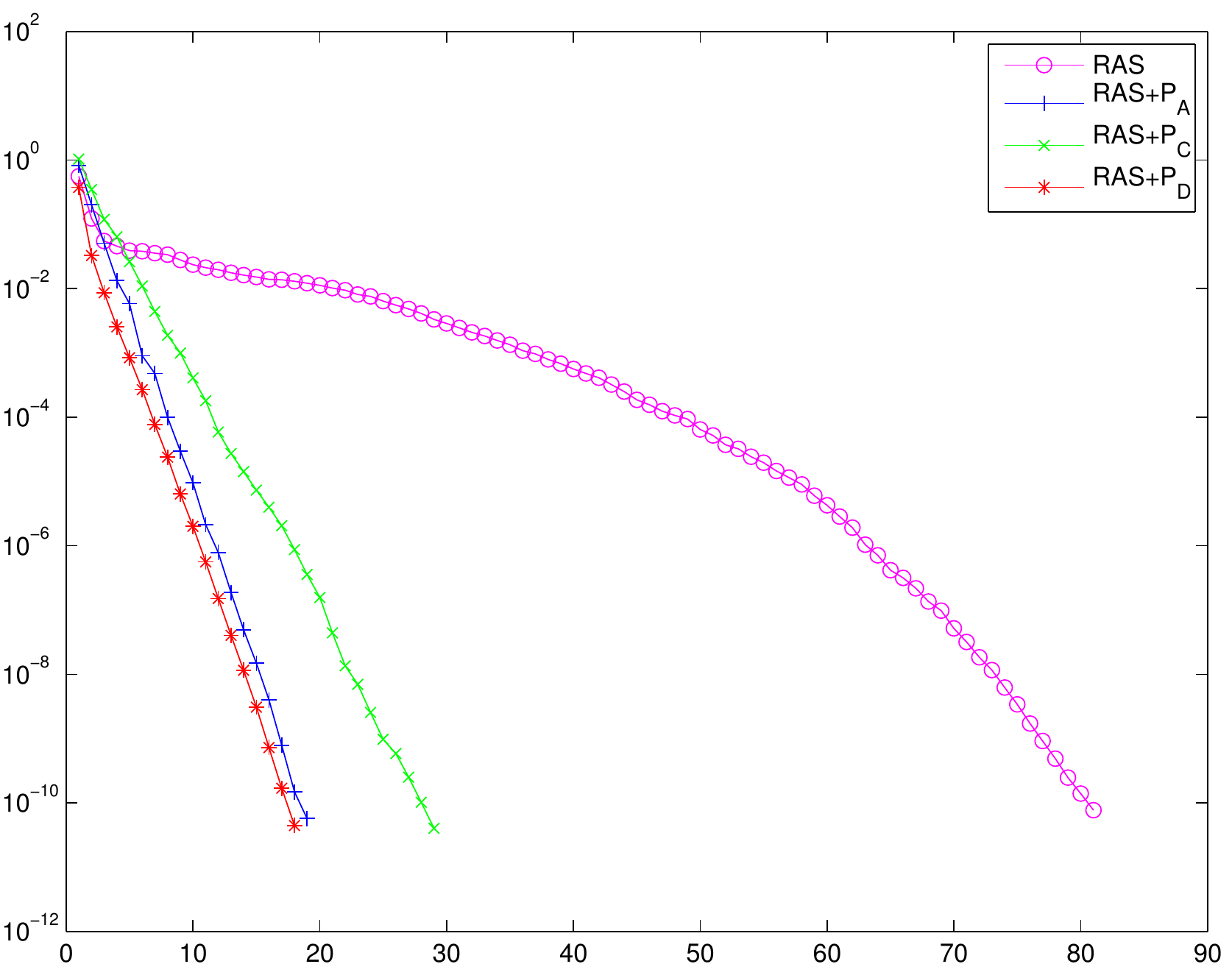}
\caption{Continuous case: 16 subdomains, 16 Ritz vectors spanning coarse grid space.}
\label{fig:bvp_c-Ritz16-part16}
\end{minipage}
\hfill
\begin{minipage}[t]{0.48\linewidth}
\centering
\includegraphics[width=\textwidth]{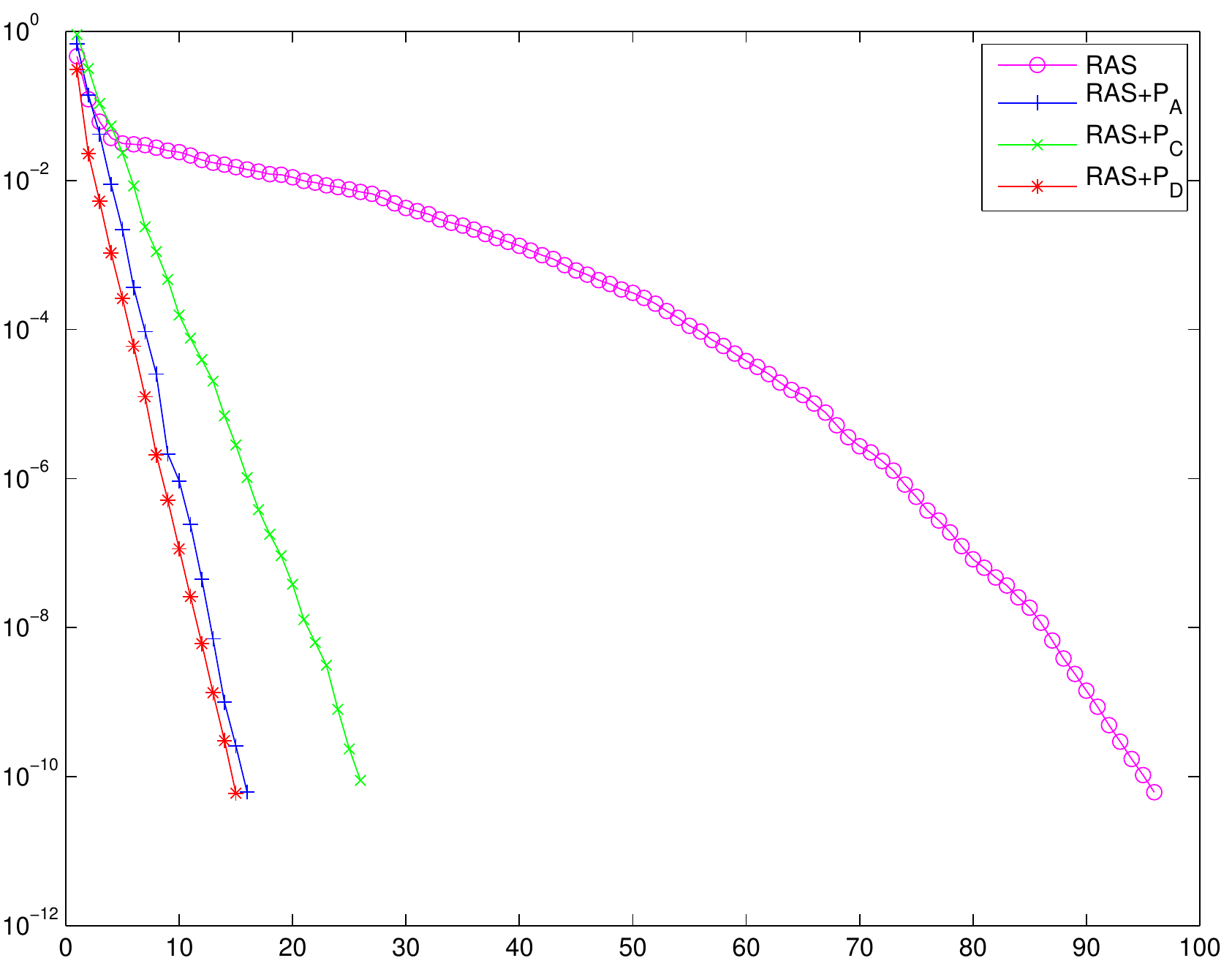}
\caption{Continuous case: 32 subdomains, 20 Ritz vectors spanning coarse grid space.}
\label{fig:bvp_c-Ritz20-part32}
\end{minipage}
\end{figure}

\begin{figure}[htbp!]
\begin{minipage}[t]{0.48\linewidth}
\centering
\includegraphics[width=\textwidth]{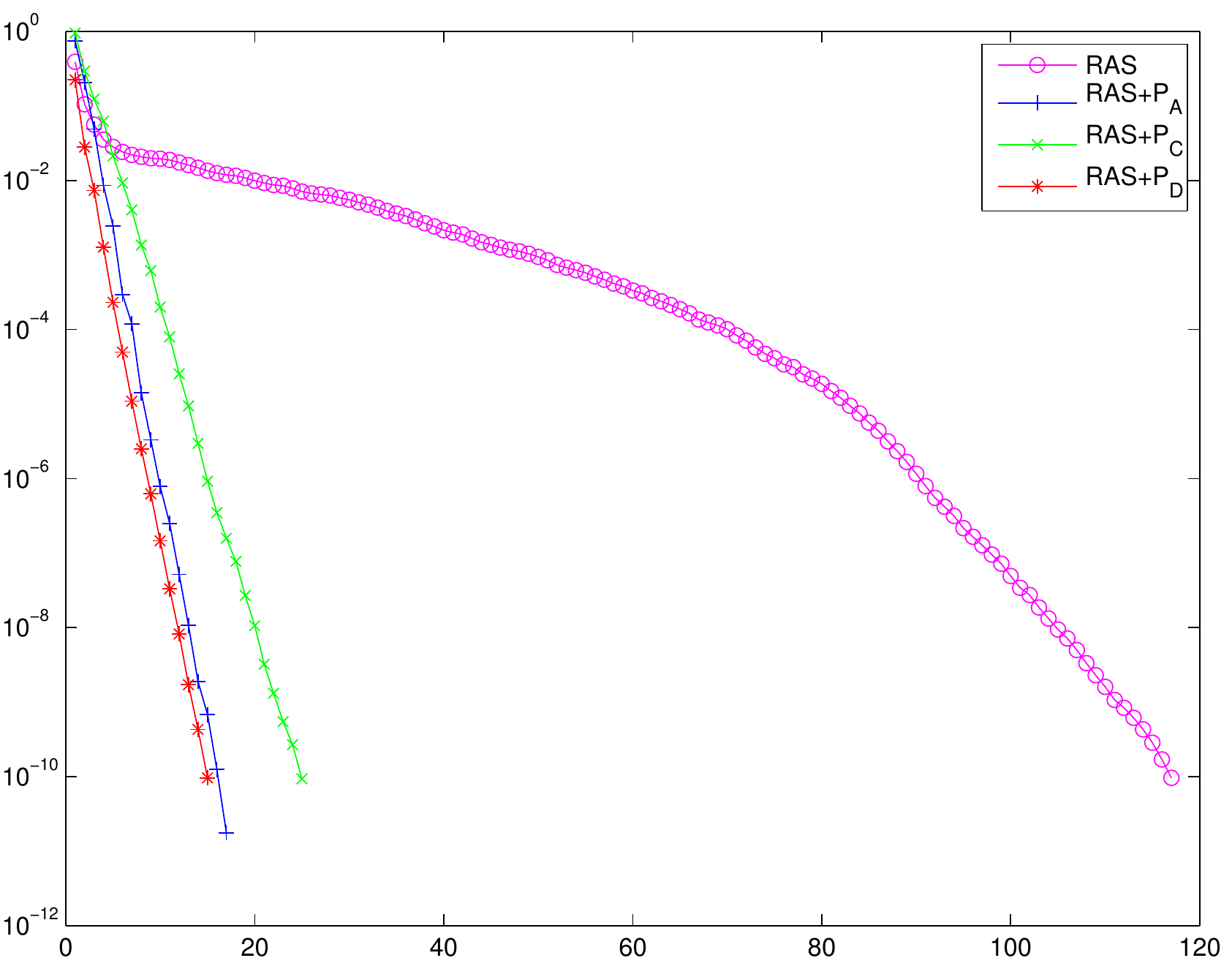}
\caption{Continuous case: 64 subdomains, 20 Ritz vectors spanning coarse grid space.}
\label{fig:bvp_c-Ritz20-part64}
\end{minipage}
\hfill
\begin{minipage}[t]{0.48\linewidth}
\centering
\includegraphics[width=\textwidth]{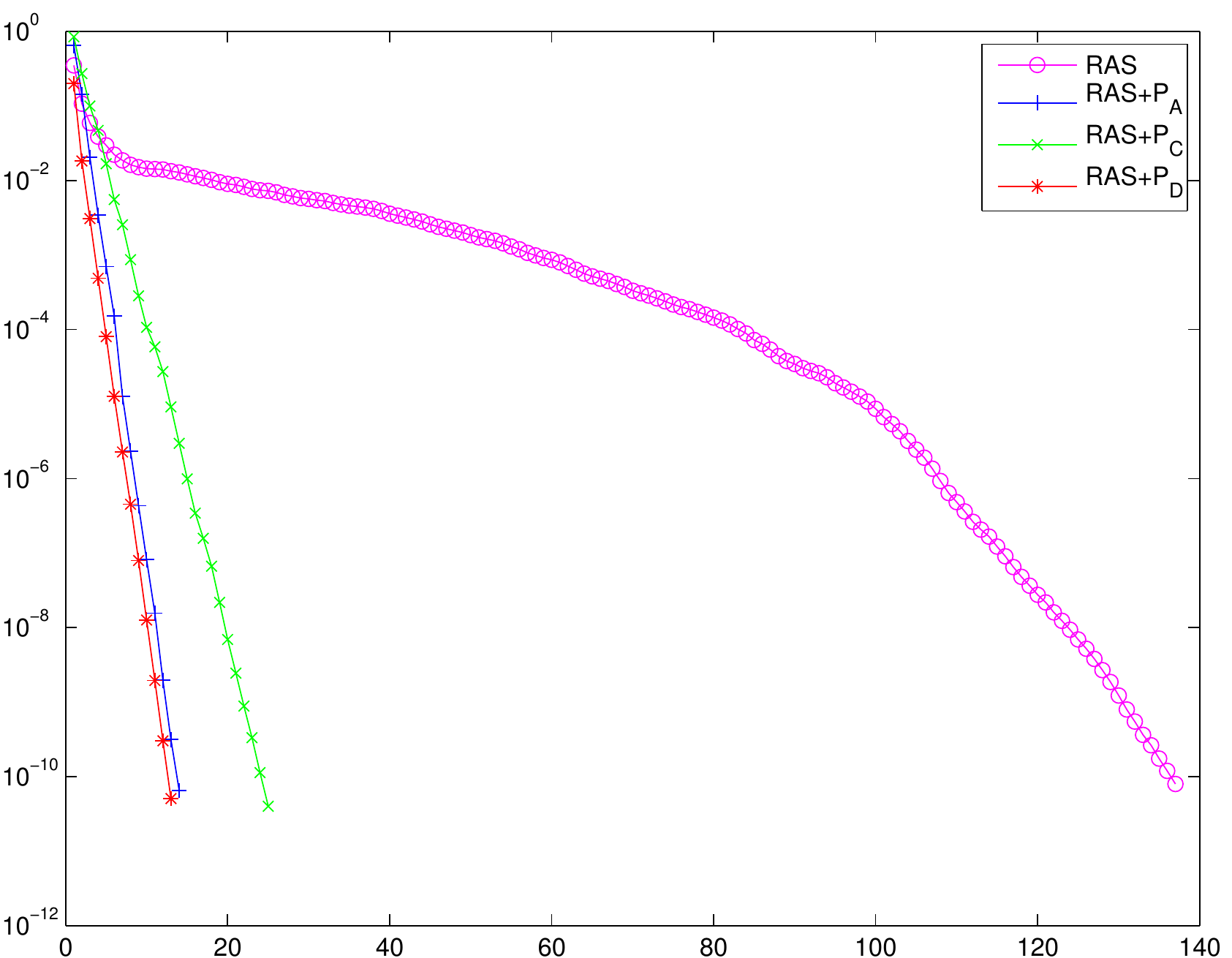}
\caption{Continuous case: 128 subdomainss, 20 Ritz vectors spanning coarse grid space.}
\label{fig:bvp_c-Ritz20-part128}
\end{minipage}
\end{figure}
Figures \ref{fig:bvp_c-Ritz16-part16}-\ref{fig:bvp_c-Ritz20-part128} plot the convergence curves for the case of continuous viscosity against the various number of subdomains. Again two-level method with RAS and the three preconditioners all outperform one-level method with only RAS. For this case, $P_D$ and $P_A$ behave almost the same on all four decompositions and outperform $P_C$.

\begin{table}[htbp!]
\caption{Maximal residual of Ritz pairs for skyscraper and continuous cases.}
\centering\small
\begin{tabular}{|c|c|c|c|c|}\hline
 Nparts     & 16 & 32 & 64 & 128 \\ \hline
 skyscraper & 1.152184e-06 & 1.973786e-04 & 2.417277e-02 & 5.017373e-02  \\ \hline
 continuous & 4.760946e-03 & 2.395625e-02 & 2.003093e-02 & 2.494041e-02  \\ \hline
\end{tabular}\label{ta:residual}
\end{table}

Table \ref{ta:residual} shows the maximal residual of Ritz pairs for both cases that are extracted from Krylov subspace when solving the system preconditioned by RAS.

Note that the size of $E$ becomes large when decomposition has 64 or 128 subdomains. As a consequence, computing $E^{-1}$ is costly and undermines gains in the number of iterations. Hence, we attempt to compute the inverse of an approximation to $E$ instead of $E^{-1}$. Assume $L$ and $U$ are factors of incomplete LU factorization(ILU) with no fill-in of $E$. Note that $E$ has a sparse structure because of the sparse structure of $Z$. $L$ and $U$ are also sparse and cheap to be inverted. Thus we replace $E^{-1}$ by $(LU)^{-1}$.

Figure \ref{fig:bvp_s-InexactE-part64} and Figure \ref{fig:bvp_s-InexactE-part128} plot the convergence curves on decompositions with 64 and 128 subdomains for skyscraper case. $P_A$ and $P_C$ are insensitive to the perturbation in $E^{-1}$, while $P_D$ fails to converge. For the continuous case, Figure \ref{fig:bvp_c-InexactE-part64} and Figure \ref{fig:bvp_c-InexactE-part128} show $P_A$ and $P_C$ behave stable on decompositions with 64 and 128 subdomains, while $P_D$ is only stable on decomposition with 64 subdomains since $LU$ is almost equal to $E$. Table \ref{ta:skycraper} and Table \ref{ta:continuous} present the distance between $LU$ and $E$ for both cases respectively.

\begin{figure}[htbp!]
\begin{minipage}[t]{0.48\linewidth}
\centering
\includegraphics[width=\textwidth]{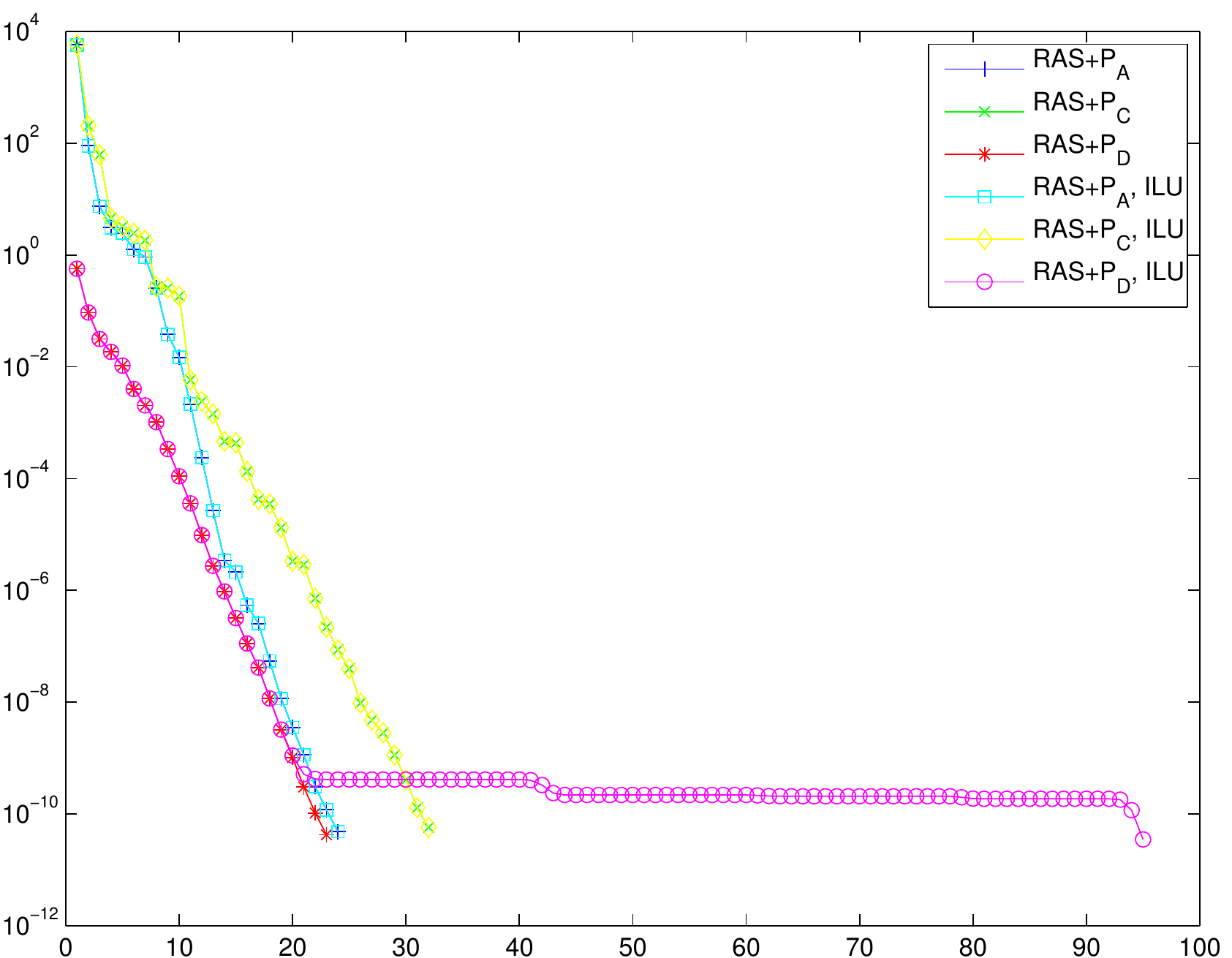}
\caption{Comparison of different projection matrix on the decomposition with 64 subdomains for skyscraper case.}
\label{fig:bvp_s-InexactE-part64}
\end{minipage}
\hfill
\begin{minipage}[t]{0.48\linewidth}
\centering
\includegraphics[width=\textwidth]{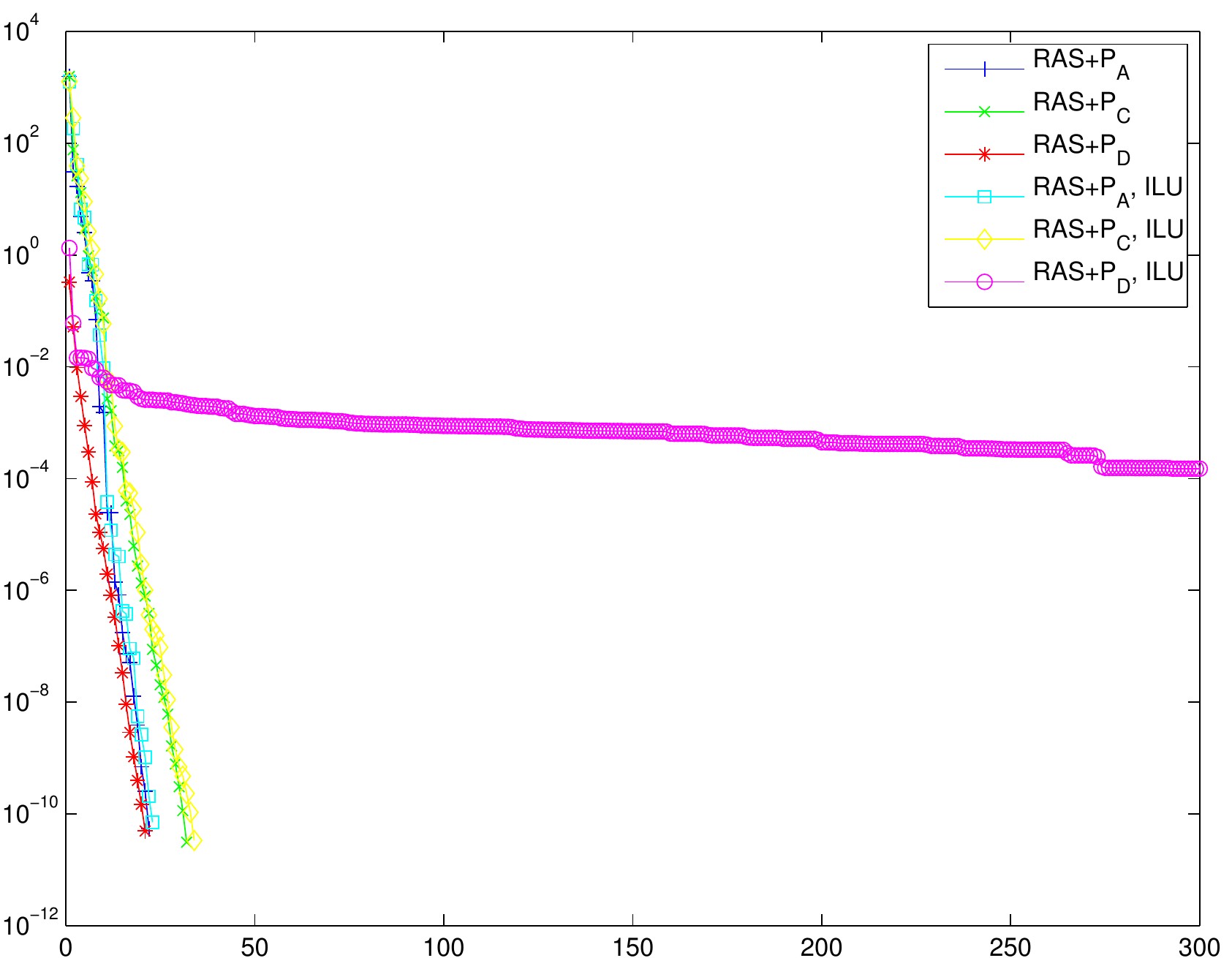}
\caption{Comparison of different projection matrix on the decomposition with 128 subdomains for skyscraper case.}
\label{fig:bvp_s-InexactE-part128}
\end{minipage}
\end{figure}

\begin{table}[htbp!]
\caption{The distance between $LU$ and $E$ for skyscraper case.}
\centering\small
\begin{tabular}{|c|c|c|}\hline
 Nparts             & 64         & 128 \\ \hline
 $\|E(LU)^{-1}-I\|_2$  & 3.8588e-08 & 8.9712e+02        \\ \hline
 $\|(LU)^{-1}E-I\|_2$  & 4.1433e-10 & 8.6292            \\ \hline
\end{tabular}\label{ta:skycraper}
\end{table}

\begin{figure}[htbp!]
\begin{minipage}[t]{0.48\linewidth}
\centering
\includegraphics[width=\textwidth]{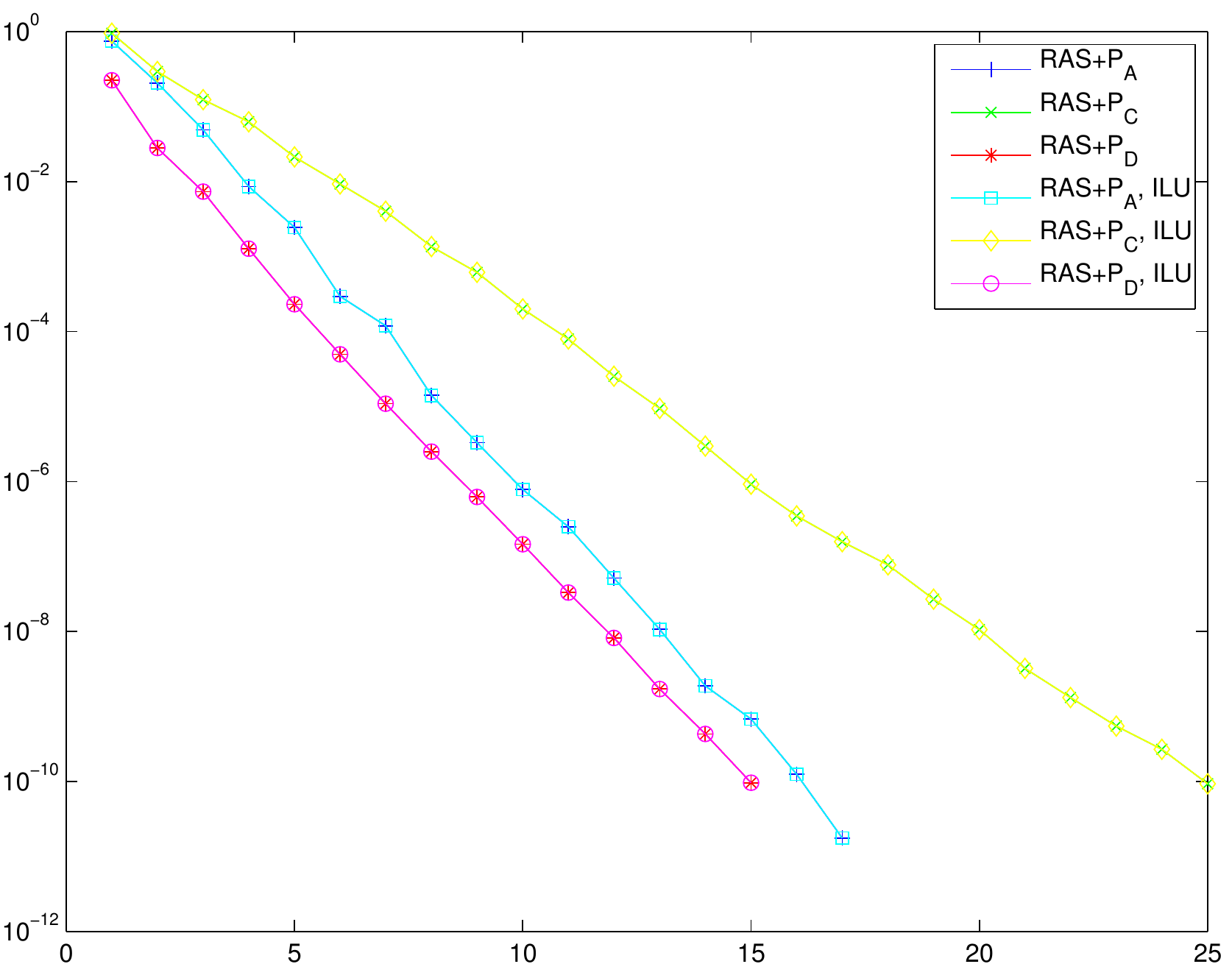}
\caption{Comparison of different projection matrix on the decomposition with 64 subdomains for continuous case.}
\label{fig:bvp_c-InexactE-part64}
\end{minipage}
\hfill
\begin{minipage}[t]{0.48\linewidth}
\centering
\includegraphics[width=\textwidth]{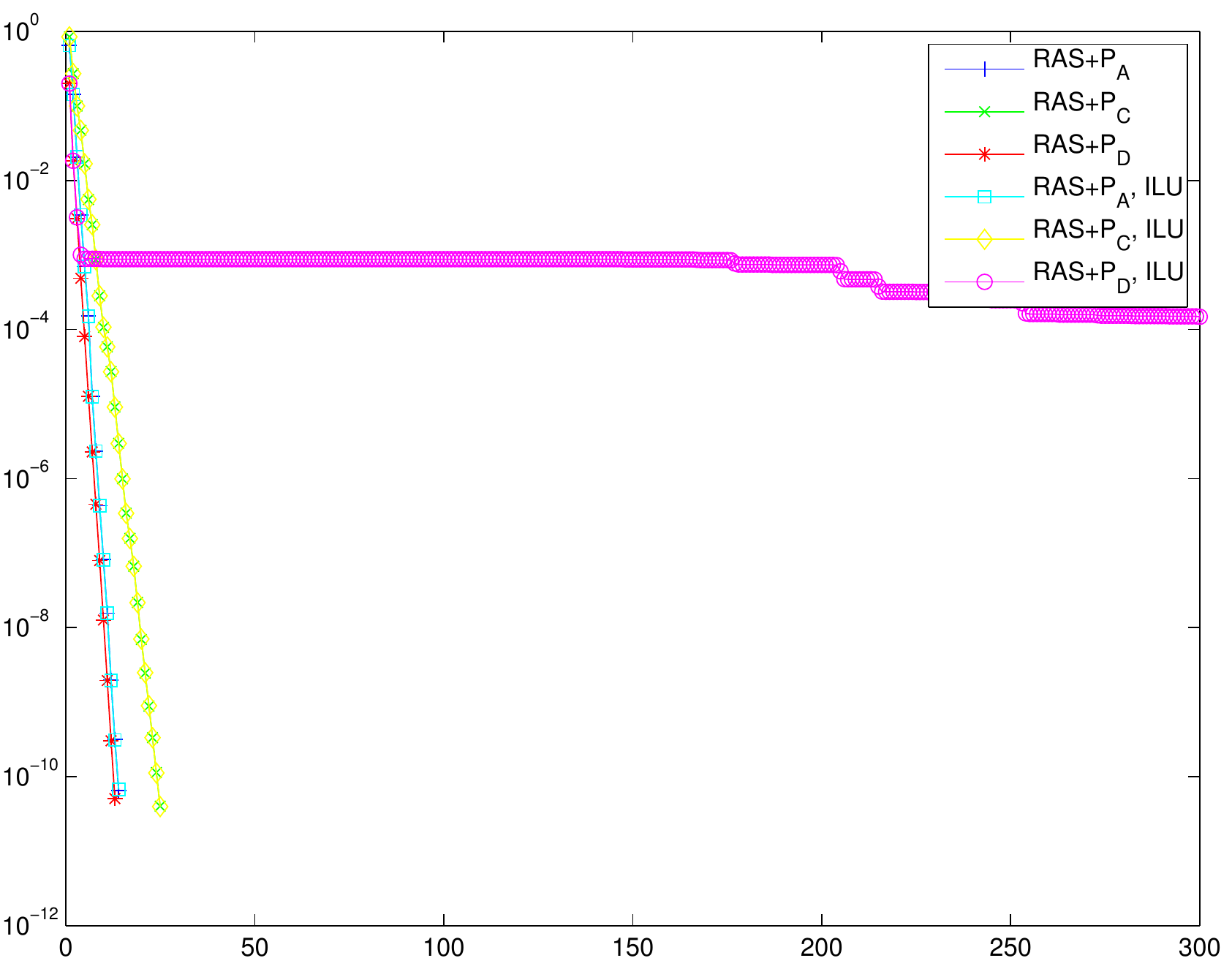}
\caption{Comparison of different projection matrix on the decomposition with 128 subdomains for continuous case.}
\label{fig:bvp_c-InexactE-part128}
\end{minipage}
\end{figure}

\begin{table}[htbp!]
\caption{The distance between $LU$ and $E$ for continuous case.}
\centering\small
\begin{tabular}{|c|c|c|}\hline
 Nparts             & 64         & 128 \\ \hline
 $\|E(LU)^{-1}-I\|_2$  & 4.1008e-14 & 5.4090e-01        \\ \hline
 $\|(LU)^{-1}E-I\|_2$  & 1.3890e-15 & 1.0215e-01        \\ \hline
\end{tabular}\label{ta:continuous}
\end{table}

\section{Conclusion}
In this paper we present a perturbation analysis on the deflation, coarse grid correction and adapted deflation preconditioners when the inexact coarse grid space and inverse of projection matrix are applied for the construction of the preconditioners. Our analysis shows that in exact arithmetic the spectrum of the system preconditioned by the three preconditioners is impacted by the angle between the exact coarse grid space and the perturbed one. Moreover, we prove that with a certain restriction the coarse grid correction and the adapted deflation preconditioners are insensitive to the perturbation in projection matrix, whereas the deflation preconditioner is sensitive. Numerical results of the different test cases emphasized the perturbation analysis. For unsymmetric linear system, similar perturbation analysis on the coarse grid correction and adapted deflation preconditioners can be developed and requires further investigation.

\end{document}